\documentclass[reqno, 11pt]{article}

\usepackage[utf8]{inputenc}	
\usepackage[T1]{fontenc}
\usepackage[english]{babel} 
\usepackage{amsmath}
\usepackage{comment,bbm,ulem}
\usepackage{mathtools}
\usepackage{mathrsfs}
\usepackage{amssymb,amsfonts}
\usepackage{amsthm}
\usepackage{array}
\usepackage{xcolor}
\usepackage{multirow}
\usepackage{todonotes}
\usepackage{bm}
\usepackage{lineno}
\usepackage[pdftex,
colorlinks=true, 
linkcolor=blue, 
citecolor=blue, 
pagebackref=false, 
]{hyperref}

\usepackage{transparent}

\usepackage{enumitem}

\usepackage{graphicx}
\usepackage{subcaption}

\usepackage{dsfont}

\usepackage{hyperref}
\usepackage{vmargin}
\setmarginsrb{3.5cm}{2cm}{3.5cm}{2cm}{0cm}{1cm}{0cm}{1.5cm}

\allowdisplaybreaks[3]

 {\everymath{\displaystyle\everymath{}}\array}%
 {\endarray}

\theoremstyle{plain}
\newtheorem{theorem}{Theorem}
\newtheorem{proposition}[theorem]{Proposition}
\newtheorem{lemma}[theorem]{Lemma}

\newtheorem{definition}[theorem]{Definition}

\theoremstyle{remark}
\newtheorem{example}[theorem]{Example}
\newtheorem{remark}[theorem]{Remark}

\newenvironment{frcseries}{\fontfamily{frc}\selectfont}{}

\def\<{{\langle}}
\def\>{{\rangle}}

\def\1Lip{1\text{-Lip}}

\def\dd{\mathrm{d}}

\def\Tr{\mathrm{Tr}}

\DeclareMathOperator*{\argmax}{arg\,max}
\DeclareMathOperator*{\argmin}{arg\,min}

\DeclareMathOperator*{\esssup}{ess\,sup}

\title{Robust mean-field games under entropy-based uncertainty}

\author{
François Delarue 
\thanks{
Université Côte d'Azur, CNRS, Laboratoire J.A. Dieudonné,   
06108 Nice, France.\\
Email: \texttt{francois.delarue@univ-cotedazur.fr}
}
\and
Pierre Lavigne
\thanks{
Université Côte d'Azur, CNRS, Laboratoire J.A. Dieudonné,   
06108 Nice, France.\\
Email: \texttt{pierre.lavigne@univ-cotedazur.fr}
}
}

\date{dedicated to Alain Bensoussan}

\begin{document}
\maketitle
\begin{abstract}
In this article, we introduce a new class of entropy-penalized robust mean-field game problems in which the representative agent is opposed to Nature. The agent’s objective is formulated as a min--max stochastic control problem, in which Nature distorts the reference probability measure at an entropic cost. As a consequence, the distribution of the continuum of agents represented by the player is given by the effective measure induced by Nature. Existence of a mean-field game equilibrium is established via a Schauder fixed point argument. To ensure uniqueness, we introduce a joint flat anti-monotonicity and displacement monotonicity condition, extending the classical Lasry--Lions monotonicity framework. Finally, we present two classes of $N$-player games for which the mean-field game limit yields $\varepsilon$-Nash equilibria.
\end{abstract}

Keywords: Risk-averse mean-field games, Quadratic backward stochastic differential equations, Monotonicity on the space of probability measures, $\varepsilon$-Nash equilibria, Entropic penalties
\vskip 4pt

MSC2020. Primary: 49N80, 91A16; Secondary: 60H10


\section{Introduction}

In this article, we introduce a class of robust (or risk-averse) mean-field game problems in which, for a given mean-field configuration, a representative player optimizes against an adversarial agent, referred to as Nature, who acts on the probability distribution by emphasizing worst-case scenarios. For a fixed control of the representative player, the mean-field configuration is defined as the marginal law of the controlled state under the probability measure induced by Nature.

\paragraph{Formulation of the problem.}
Let $[0,T]$ be a finite time horizon and let $(\Omega,\mathcal{F},\mathbb{P})$ be a probability space supporting a $d$-dimensional Brownian motion $W=(W_t)_{0\le t\le T}$ and an independent $\mathbb{R}^n$-valued random variable $\eta$, representing the initial state of the representative player. Here, $n\in\mathbb{N}^\star$ denotes the state dimension and $d\in\mathbb{N}^\star$ the dimension of the driving noise. The $\mathbb{P}$-complete filtration generated by $(\eta,W)$ is denoted by $\mathbb{F}=(\mathcal{F}_t)_{0\le t\le T}$.

Given a mean field coupling $\mu$, viewed as an element of the space $\mathcal{M}(\mathbb{R}^d)$ of non-negative (possibly non-normalized) measures on $\mathbb{R}^d$, the representative agent seeks to minimize a robust (risk-averse) objective functional of min--max type:
\begin{equation} 
\tag{MinMax$[\mu]$}\label{pb:optim-mfg}\begin{split}
    &\inf_{\psi \in \mathcal{A}} \sup_{q \in \mathcal{Q}} \mathcal{J}[\mu](\psi,q), 
    \\
    \textrm{\rm where} \quad &\mathcal{J}[\mu](\psi,q)
    \coloneqq
    \mathbb{E}\!\left[q_T g(X_T^\psi,\mu) + \int_0^T q_s \ell(s,\psi_s)\,\mathrm{d}s \right]
    - \mathcal{S}(q).
    \end{split}
\end{equation}
Nature optimizes over a flow $q=(q_t)_{t\in[0,T]}$ of random non-normalized densities, while the representative player optimizes over a control $\psi$, with associated state process $X^\psi=(X_t^\psi)_{t\in[0,T]}$. The functional $\mathcal{S}(q)$, defined precisely below, is referred to as a generalized entropy, as it extends the classical relative entropy with respect to $\mathbb{P}$. The functions $\ell:\Omega\times[0,T]\times\mathbb{R}^n\to\mathbb{R}$ and $g:\mathbb{R}^n\times\mathcal{M}(\mathbb{R}^n)\to\mathbb{R}$ denote respectively the running cost and the mean-field terminal cost of the representative player. In this setting, min--max equilibria are sought within the class of open-loop controls.
The mean field consistency condition requires finding a measure $\mu$ such that, if $(q,\psi)$ is a saddle point of \eqref{pb:optim-mfg} (whose existence and uniqueness are ensured under the assumptions stated below), then $\mu$ coincides with the law of $X_T^\psi$ under the effective measure $q_T\mathbb{P}$ selected by Nature.

Admissible processes $q$ are assumed to admit the representation
\begin{equation}
\label{eq:q:explicit:factorization}
    q_t = e^{\int_0^t Y^\star_s\dd s } \mathcal{E}_t\biggl(\int_0^\cdot Z^\star_s \cdot \dd W_s\biggr),
   \quad t \in [0,T],
\end{equation}
  where 
  $({\mathcal E}_t(M) = \exp(M_t - \tfrac12 \langle M
\rangle_t))_{t \in [0,T]}$ denotes the Doléans-Dade exponential of
a local martingale $(M_t)_{t \in [0,T]}$ with 
$(\langle M\rangle_t)_{t \in [0,T]}$ as bracket,     
  $Y^\star=(Y_t^\star)_{t \in [0,T]}$
    and $Z^\star=(Z_t^\star)_{t \in [0,T]}$
    are two ${\mathbb F}$-progressively measurable processes
    with 
    values in ${\mathbb R}$ and ${\mathbb R}^d$, 
    respectively, such that, ${\mathbb P}$-almost surely, 
    $\int_0^T \vert Y_s^\star \vert \dd s$ and 
    $\int_0^T \vert Z_s^\star \vert^2 \dd s$ are finite. 
    In particular, $(q_t)_{t \in [0,T]}$ satisfies the equation
         \begin{align}
  q_t = 1 + \int_0^t q_s Y_s^\star \dd s + \int_0^t q_s Z_s^\star \cdot \dd W_s,
\quad t \in [0,T].  \label{eq:q} 
  \end{align}   
    The generalized entropy of $q$ is defined as 
    \begin{equation}
    \label{eq:S}
    \mathcal{S}(q) \coloneqq \mathbb{E}\left[ \int_0^T q_s  f^{\star}(s,Y^\star_s,Z^\star_s) \dd s  \right]
\end{equation}
and is required to be finite. The function $f^\star \colon \Omega \times [0,T] \times \mathbb{R} \times \mathbb{R}^d \to \mathbb{R}$ 
is the ‘convex dual' driver (as it it the
Legendre-Fenchel transform of some $f$ introduced in the assumptions below). 
    When 
    $f^\star(s,y^\star,z^\star) = \tfrac12 \vert z^\star \vert^2$
    and $Y^\star \equiv 0$ in \eqref{eq:q:explicit:factorization},
    ${\mathcal S}(q)$ coincides with the standard relative entropy 
${\mathbb E}[q_T \ln(q_T)]$. 
The set of admissible processes $q$ is denoted $\mathcal{Q}$.

The admissible control set $\mathcal{A}$ consists of all $\mathbb{F}$-progressively measurable, $\mathbb{R}^n$-valued processes $\psi=(\psi_t)_{t\in[0,T]}$ satisfying
\begin{equation}
\label{eq:expo:bound:psi}
\mathcal{S}^\star(\psi) < +\infty,
\qquad
\mathcal{S}^\star(\psi)
\coloneqq
\sup_{q\in\mathcal{Q}}
\left\{
\mathbb{E}\!\left[\int_0^T q_s |\psi_s|^2\,\mathrm{d}s\right]
-
\gamma\,\mathcal{S}(q)
\right\},
\end{equation}
where the constant $\gamma>0$ is specified below in accordance with the assumptions on the model coefficients.
Although this definition may appear technical at first glance, it in fact captures the duality between the state of Nature and that of the representative player. In particular, when $\mathcal{S}$ coincides with the relative entropy with respect to $\mathbb{P}$, condition \eqref{eq:expo:bound:psi} echoes the Donsker--Varadhan duality formula and effectively enforces the existence of an exponential moment for $\int_0^T |\psi_s|^2\,\mathrm{d}s$. The role played by $\mathcal{S}^\star$ in the analysis of problem \eqref{pb:optim-mfg} has been highlighted in our parallel work \cite{DelarueLavigne}, where the present framework is adapted to robust mean field control.
For a given control $\psi \in \mathcal{A}$, the state $X^\psi=(X_t^\psi)_{t \in [0,T]}$ of the representative player is the solution to
\begin{equation}
\label{eq:intro:X}
    \dd X_t = b(t,X_t,\psi_t) \dd t + \sigma(t,\psi_t) \dd W_t, \quad X_0 = \eta,
\end{equation}
where the drift $b \colon \Omega \times [0,T] \times \mathbb{R}^n \times \mathbb{R}^n \to \mathbb{R}^n$ and the volatility $\sigma \colon \Omega \times [0,T] \times \mathbb{R}^n \to \mathbb{R}^{n \times d}$ are possibly random.
Implicitly, the coefficients $b$ and $\sigma$ are assumed to be $\mathbb{F}$-progressively measurable. Their precise regularity and structural assumptions are specified later in the paper. In particular, although the state equation \eqref{eq:intro:X} will ultimately be taken to be linear, we keep its general form here for expositional purposes.

We now provide a more detailed description of the equilibrium condition.
Let $q$ and $\psi$ denote the optimal 
state of Nature and the optimal control of the representative agent
respectively 
(see 
Theorem 
\ref{theorem:SMP}
for the existence and uniqueness  of 
such a saddle point). 
In the mean-field framework, the representative agent is assumed to be typical of a continuum of statistically identical and independent agents, all playing the same game. ln particular, the mean-field equilibrium condition requires that the coupling measure $\mu$ coincides with the law of the terminal state $X_T^\psi$ but under the \textit{effective} measure 
 induced by Nature. 
Formally,
the latter writes $q {\mathbb P} : {\mathcal F} \ni A \mapsto {\mathbb E}[ q {\mathds 1}_A]$, and the 
law of $X_T^\psi$ under $q {\mathbb P}$ is 
$q {\mathbb P} \circ (X_T^\psi)^{-1}$ (which we also write 
$(q {\mathbb P})_{X_T^\psi}$). This leads to the fixed point condition:
\begin{equation} \tag{MFG-eq}\label{eq:equilibrium-condition}
    \mu = ( q {\mathbb P}) \circ (X^\psi_T)^{-1}.
\end{equation}
We refer to such a mean-field equilibrium as \textit{consistent}, meaning that the mean-field coupling observed by each agent at equilibrium is determined by the probability measure induced by Nature. In this approach, the agent’s risk sensitivity is fully reflected in the mean-field coupling.

Alternatively, one may consider the classical fixed point condition
\[
\mu = \mathbb{P} \circ (X_T^\psi)^{-1}.
\]
However, 
under this prescription, the equilibrium is \textit{inconsistent}: the representative agent remains risk-averse with respect to their own idiosyncratic noise, yet anticipates that the aggregate mean-field will materialize under the reference measure $\mathbb{P}$. This inconsistency arises from a mismatch between the agent's risk perception and the formation of the mean-field interaction.

In the rest of the article, 
the mean-field game problem thus consists in finding a triple $(q,\psi,\mu) \in \mathcal{Q} \times \mathcal{A} \times \mathcal{M}(\mathbb{R}^n)$ solving \eqref{pb:optim-mfg}-\eqref{eq:equilibrium-condition}, 
which may be summarized as
 \begin{equation} \tag{MFG} \label{pb:mfg}
     \mathcal{J}[\mu](\psi,q) = \inf_{\psi' \in \mathcal{A}} \sup_{q' \in \mathcal{Q}} \mathcal{J}[\mu](\psi',q'), \quad \mu = ( q {\mathbb P}) \circ (X^\psi_T)^{-1}.
\end{equation}

\paragraph{FBSDE formulation} In the existing literature, mean-field game equilibria are typically characterized by a system of forward-backward partial differential equations (PDEs)---specifically two  {Fokker-Planck} and  {Hamilton-Jacobi-Bellman} equations, see \cite{LL06cr1,LL06cr2}---or, from a probabilistic standpoint, a system of forward-backward stochastic differential equations (FBSDEs), see \cite{carmona2018probabilistic-v1}.
In this work, we adopt the probabilistic perspective. Under this formulation, the mean-field game FBSDE is viewed as the first-order system describing the optimal control problem of a representative agent interacting with a prescribed population distribution. The equilibrium is established via a fixed-point condition, which requires that this distribution coincides with the law induced by the agent's optimal strategy
under Nature's density process. 
Within our framework, the FBSDE 
associated with the representative agent is coupled with another FBSDE, describing Nature's optimal control. The resulting equilibrium must therefore account for the simultaneous optimization of the agent and Nature, coupled with the aggregate consistency stemming from the mean-field interaction.

Given an arbitrary   measure $\mu \in {\mathcal M}({\mathbb R}^n)$,
candidate for being a mean-field equilibrium, the two FBSDEs are driven by 
the following two pre-Hamiltonians, associated with the representative player and Nature
respectively:
\begin{equation}
    \label{eq:def:H:F}    
    \begin{split}
    H(t,x,\psi,p,k,q) & := q \ell(t,\psi) + p \cdot b(t,x,\psi)+ \Tr (k \sigma^{\top}(t,\psi) ), \\
F(t,q,y^\star,z^\star,y,z,\psi) & := q\left(yy^\star + z\cdot z^\star - f^\star(t,y^\star,z^\star) + \ell(t,\psi)\right).
\end{split}
\end{equation}
Given $q \in \mathcal{Q}$, we say that a tuple $(\psi,p,k,X)$ satisfies the first order condition \eqref{optim:condition-primal} for the representative player problem if $(\psi,p,k,X)$ is a solution to
\begin{equation} \tag{RP$[\mu]$}
    \left\{ \begin{array}{rll}
        - \dd p_t  & =   \nabla_x H(t,X_t,\psi_t,p_t,k_t,q_t) \dd t - k_t \dd W_t, & 
        p_T = q_T \nabla g(X_T^\psi, \mu), \\[0.5em]
        \dd X_t & = b(t,X_t,\psi_t) \dd t + \sigma(t,\psi_t) \dd W_t,  & X_0 = \eta,\\[0.5em]
        \psi_t & \in \argmin_{\alpha} H(t,X_t,\alpha,p_t,k_t,q_t), & \dd \mathbb{P}  \otimes \dd t \text{-a.e.}
    \end{array} \right.
\end{equation}
The first equation is interpreted as the adjoint equation for the representative player, the second equation as the state equation, and the last equation as the optimality condition. Because the last equation creates a coupling between the first two equations, the system above is an FBSDE.

Given $\psi \in \mathcal{A}$, we say that a tuple $(Y,Z,q)$ satisfies the first order condition \eqref{optim:condition-dual} for Nature problem if  $(Y,Z,q)$ is a solution to 
\begin{equation*}  \tag{N$[\mu]$}
    \left\{ \begin{array}{rll}
        - \dd Y_t & =  \partial_q F(t,q_t,Y^{\star}_t,Z^{\star}_t,Y_t,Z_t,\psi_t) \dd t - Z_t \cdot \dd W_t,  & Y_T = g( X_T^\psi, \mu), \\[0.5em]
         \dd q_t &  = q_t Y^\star_t  \dd t +  q_t Z^\star_t  \cdot \dd W_t, &
        q_0  = 1,\\[0.5em]
        (Y^\star_t,Z^\star_t) & \in \argmax_{(Y^{\star \prime},Z^{\star \prime})}F(t,q_t,Y^{\star \prime},Z^{\star \prime},Y_t,Z_t,\psi_t),  & \dd \mathbb{P}  \otimes \dd t\text{-a.e.}
    \end{array} \right.
\end{equation*}
The first equation is interpreted as the adjoint equation to Nature' state. 
The process 
$Y$ describes the time instantaneous value of the representative player. Indeed, when $\psi$ is optimal for the representative player, $Y$ can be seen as the solution to a (risk averse) dynamic programming principle for the representative player. When $(Y^\star,Z^\star) \equiv 0$ (and thus $q \equiv 1$),
and (say) $f^\star(s,0,0)=0$, we clearly recover the standard dynamic programming principle.  
The second equation describes the dynamics of the control variable, and the last equation is the optimality condition. This system of equations is also an FBSDE for the same reason as the previous system. 
The two FBSDEs are obviously coupled. 

The FBSDE characterizing the Nash equilibria is given by \eqref{optim:condition-primal}-\eqref{optim:condition-dual} complemented with the equilibrium condition \eqref{eq:equilibrium-condition}.

In the benchmark case where $f^\star(s, y^\star, z^\star) = \frac{1}{2} |z^\star|^2$ and $f(s, y, z) = \frac{1}{2} |z|^2$, the FBSDE for $Y$ becomes quadratic. This structure significantly complicates the solvability of the coupled systems, particularly when the terminal reward $g$ is an unbounded function of the state.
Under these conditions, the solvability of the quadratic BSDE satisfied by $Y$ appears to lie beyond the scope of standard results in the literature, such as those found in \cite{briand2006bsde, briand2008quadratic}. This difficulty prompted a dedicated investigation into the properties of this BSDE in our companion work \cite{DelarueLavigne}, where we provide a tailored analysis leveraging the specific min-max structure of problem \eqref{pb:optim-mfg}.
Broadly speaking, for a fixed measure $\mu$, we establish existence and uniqueness by exploiting the underlying concavity in the variable $q$ and convexity in the variable $\psi$. These structural properties are further utilized in our approach to resolving the fixed-point condition \eqref{eq:equilibrium-condition}.

A typical situation where these conditions arise is that of a financial investor seeking to maximize utility while being subject to trading costs. Assuming the market consists of $n$ assets, each evolving according to the dynamics
\begin{equation*}
    \frac{\dd S_t^i}{S_t^i} = c_t^i \dd t + \sigma_t^i \dd W_t,
\end{equation*}
where $W$ is a $d$-dimensional noise process and the (potentially random) coefficients $c^i$ and $\sigma^i$ are of appropriate dimensions, the investor's self-financing portfolio $X^\psi$ evolves according to the equation
\begin{equation*}
    \dd X_t^\psi = \sum_{i=1}^n \psi_t^i \frac{\dd S_t^i}{S_t^i}, \quad X_0 = 1,
\end{equation*}
where the initial condition is arbitrarily chosen to be unitary. The problem of the risk-averse investor, expressed in a min-max form, is given by
\begin{equation} \label{pb:inve-risk-averse}
    \sup_{q \in \mathcal{Q}} \inf_{\psi\in \mathcal{A}} \mathcal{J}(q,\psi),
\end{equation}
where 
\begin{equation*}
    \mathcal{J}(q,\psi) = \mathbb{E}^{\mathbb{Q}}\left[ g(X_T^\psi ) + \frac{1}{2} \int_0^T |\psi_s|^2 ds \right] - \gamma \mathcal{H}(\mathbb{Q} \vert \mathbb{P}), \quad \mathbb{Q} = q_T \mathbb{P},
\end{equation*}
with $q_T = \mathcal{E}_T(\int_0^\cdot Z_s^\star dW_s)$, and $g$ denotes a payoff function and 
${\mathcal H}({\mathbb Q}\vert {\mathbb P})$
the relative entropy of 
${\mathbb Q}$ with respect to 
${\mathbb P}$. 

Under this formulation, given a probability measure $\mathbb{Q}$ equivalent to $\mathbb{P}$, the investor optimizes the portfolio return while incurring a quadratic trading cost. Conversely, for a given investor strategy, Nature selects the worst-case probability measure $\mathbb{Q}$, subject to an entropic penalty. The parameter $\gamma > 0$ models the investor's level of risk aversion.
In a mean-field regime, the payoff may take the form $g(X_T^\psi - \lambda \bar{\mu})$, where $\lambda$ represents an interaction parameter and $\bar{\mu}$ denotes the mean of $\mu$, representing the average wealth under the effective probability measure.

\paragraph{Literature.}Mean-field games (MFGs) are competitive problems involving a continuum of agents whose interactions occur through a mean-field functional. They naturally arise as the limit of large, symmetric, and anonymous finite-player games, much like stochastic mean-field control problems. The theory was independently introduced in \cite{huang2007large} and \cite{LL06cr1,LL06cr2} and has since been extensively developed; see \cite{bertucci2019some,BHP-schauder,cardaliaguet2019master,cardaliaguet2015second,carmona2018probabilistic-v1,carmona-delarue-lacker}. MFGs have found numerous applications, including economics and finance \cite{achdou2023simple,cardaliaguet2018mean,feron2021price,mouzouni2019topic}, environmental studies \cite{kobeissi2024tragedy,lavigne2023decarbonization}, and electricity markets \cite{alasseur2020extended}. For a comprehensive exposition of the theory, see the monograph \cite{carmona2018probabilistic-v1}.

The classical theory of mean field games (MFGs) primarily considers risk-neutral agents minimizing expected costs. Extending this framework to account for risk aversion has been the focus of several lines of research, each introducing different ways to capture the agents' attitudes toward uncertainty. In risk-sensitive MFGs \cite{moon2016linear,saldi2020approximate,tembine2013risk}, agents optimize exponential or variance-sensitive criteria, which penalize high variability in costs (see also 
\cite{bensoussan2017risk} for the mean-field control analogue problem). More generally, risk-averse MFGs \cite{cheng2023risk,escribe2024mean,lavigne2020} incorporate abstract risk measures into the cost functional, allowing a wide class of preferences beyond variance-based penalties. Finally, robust or worst-case MFGs, often inspired by $H^\infty$ control, introduce an explicit adversarial player (Nature) that acts to worsen the representative agent's outcome \cite{bauso2016robust,zaman2024robust}, capturing ambiguity and model uncertainty. These different approaches reflect complementary ways to model agent sensitivity to risk and uncertainty in large populations.

The robust mean-field game studied here is closely related to the risk-sensitive framework: by the Gibbs-variational (Donsker--Varadhan) principle, minimizing an exponential cost is equivalent to a min--max game in which Nature selects a worst-case measure penalized by relative entropy. That said, unlike the standard risk-sensitive setting, we model Nature explicitly as acting on the weighting of events, which allows the equilibrium condition \eqref{eq:equilibrium-condition} to be defined under the effective measure. To our knowledge, this explicit incorporation of the effective measure into the equilibrium definition is novel and provides a new perspective on robust mean-field equilibria.

\paragraph{Contributions.}Beyond the model itself, which we find interesting, we contribute the following results. First, based on Schauder’s theorem, we establish a general existence result for equilibria (see Theorem \ref{thm:existence}). Compared to standard MFGs, the proof requires careful treatment of Nature’s state. Continuity with respect to the state of Nature is obtained by combining entropy-type inequalities, established under ad-hoc convexity assumptions in \cite{DelarueLavigne}, with Pinsker’s inequality, which ultimately controls the total variation of Nature’s state.
Next, we identify a general uniqueness criterion (see Proposition \ref{prop:uniqueness}), which can be seen as an analogue of the Lasry--Lions monotonicity conditions (or displacement monotonicity in certain cases) in the risk-neutral setting. When the game is derived from a potential, this criterion reduces to a joint condition of flat concavity and displacement convexity for the potential. We also provide examples of non-potential games where the condition holds.
Finally, we discuss the connection with finite-player models. This question is subtle, since the law of large numbers underlying the derivation of the mean-field model is perturbed by Nature’s behavior. We present two approaches to show how the mean-field regime can emerge asymptotically and quantify to what extent the asymptotic equilibrium induces approximate equilibria (Lemmas \ref{lem:23}, \ref{lem:27}, and \ref{lem:31}). As in the rich literature on convergence in MFG theory, the passage from finite-player games to the continuum remains a challenging problem and certainly calls for further study.

\paragraph{Organization of the article.}
The article is organized as follows. In Section \ref{sec:notations}, we useful introduce notations and definitions. In Section \ref{sec:robust}, we present the stochastic maximum principle recently obtained in \cite{DelarueLavigne}, which allows us to handle \eqref{pb:optim-mfg} when $\mu$ is fixed. Section \ref{se:mfg} addresses the solvability of the mean-field game, providing both existence and uniqueness results. Finally, Section \ref{se:limiting} is dedicated to the passage from two forms of finite-player game to the mean-field limit.

Several results in the text are directly taken from \cite{DelarueLavigne}; nevertheless, we have written the exposition to maintain a smooth and coherent flow.

\section{Notations and definitions} \label{sec:notations}
This section introduces the notation used throughout the paper.

\paragraph{Spaces of random variables and processes.}
We work on the same filtered, complete probability space $(\Omega, \mathcal{F}, \mathbb{F}, \mathbb{P})$ as in the introduction. When another probability measure is used, this will be indicated explicitly in the notation of the corresponding spaces of random variables or processes; for instance, we write $L^p(\cdot, \mathbb{Q})$, which is the first space defined in the list below.
\vskip 4pt
For each $t \in [0,T]$,
we denote by $L^0(\mathcal{F}_t,\mathbb{R}^d)$ the set of $\mathbb{R}^d$ valued and $\mathcal{F}_t$-measurable random variables 
(r.v.'s in short). And then, for each $p>0$, we define the sets

\begin{itemize}
\item $L^p(\mathcal{F}_t,\mathbb{R}^d)$  of r.v.'s 
$X \in L^0(\mathcal{F}_t,\mathbb{R}^d)$ s.t. $ \| X \|_{L^p(\mathcal{F}_t,\mathbb{R}^d)} \coloneqq \mathbb E[|X|^p]<+\infty;$
\item $L^\infty(\mathcal{F}_t,\mathbb{R}^d)$  of r.v.'s $X \in L^0(\mathcal{F}_t,\mathbb{R}^d)$
s.t. $\| X \|_{L^\infty(\mathcal{F}_t,\mathbb{R}^d)} \coloneqq 
\underset{\omega \in \Omega}{\esssup} \underset{i \in \{1,\ldots,d\}}{\sup} |X^i(\omega)|$ $ < + \infty.$
\end{itemize}

We denote by $L^0(\mathbb F,\mathbb{R}^d)$ the space of $\mathbb F$-progressively measurable random processes (r.p.'s in short) with values in $\mathbb{R}^d$, and by $S^0({\mathbb F},{\mathbb R}^d)$ the subset of $L^0({\mathbb F},{\mathbb R}^d)$ comprising processes with continuous trajectories. 
Given $p>0$, we define the sets
\begin{itemize}
    \item $L^{p}(\mathbb F,\mathbb{R}^d)$ of r.p.'s $X \in L^0(\mathbb F,\mathbb{R}^d)$
    s.t. $\| X \|_{L^{p}(\mathbb F,\mathbb{R}^d)} \coloneqq \mathbb E\left[ \left(\displaystyle \int_0^T |X_t|^p \dd t \right)^{1/p} \right]<+\infty,$
    \item $M^{p}(\mathbb F,\mathbb{R}^d)$ of r.p.'s $X \in L^0(\mathbb F,\mathbb{R}^d)$
    s.t. $\| X \|_{M^{p}(\mathbb F,\mathbb{R}^d)} \coloneqq \mathbb E\left[ \left(\displaystyle \int_0^T |X_t|^2 \dd t \right)^{p/2} \right]<+\infty,$
    \item $L^\infty(\mathbb F,\mathbb{R}^d)$  of 
    r.p.'s
    $X \in L^0(\mathbb F,\mathbb{R}^d)$ s.t.  $\| X \|_{L^\infty(\mathbb F,\mathbb{R}^d)} \coloneqq \underset{t \in [0,T]}{\sup} \| X_t\|_{L^\infty(\mathcal{F}_t,\mathbb{R}^d)} < + \infty$,
     \item $S^p(\mathbb F,\mathbb{R}^d)$ of r.p.'s $X \in S^0(\mathbb F,\mathbb{R}^d)$
    s.t. $\|X\|_{S^p(\mathbb F,\mathbb{R}^d)}  \coloneqq  \mathbb E\left[
    \underset{t \in [0,T]}{\sup} |X_t|^p \right]<+\infty$,
    \item $D({\mathbb F},{\mathbb R}^d)$ of r.p.'s $X \in S^0({\mathbb F},{\mathbb R}^d)$ such that the family 
    $(\vert X_{\tau}\vert)_{\tau}$, with 
    $\tau$ running over the set of $[0,T]$-valued ${\mathbb F}$-stopping times, is uniformly integrable.
\end{itemize}

For more 
details on the  class 
$D({\mathbb F},{\mathbb R}^d)$, we refer to \cite[Definition 20]{dellacherie:meyer:b}.
For each space defined above, we omit the notation $\mathbb{R}^d$ when $d = 1$.

\paragraph{Spaces of measures.} We call 
$\mathcal{P}(\mathbb{R}^n)$ the set of probability measures on ${\mathbb R}^n$, and $\mathcal{M}(\mathbb{R}^n)$ the set of finite non-negative measures on $\mathbb{R}^n$. For any $p \geq 1$,  we define the sets
\begin{itemize}
    \item $\mathcal{P}_p(\mathbb{R}^n)$ of $\mu \in \mathcal{P}(\mathbb{R}^n)$ s.t. $\int_{{\mathbb R}^n} \vert x \vert^p  \dd \mu(x) < +\infty$, 
    \item $\mathcal{M}_p(\mathbb{R}^n)$ of $\mu \in \mathcal{M}(\mathbb{R}^n)$ s.t. $M_p(\mu):=\int_{\mathbb{R}^d} \vert x \vert^p  \dd \mu(x) < +\infty$.
\end{itemize}

For any finite measure $\mathbb{Q}$ on $\Omega$ and any measurable mapping $X \colon \Omega \to {\mathbb R}^n$, we denote by $\mathbb{Q}_X = \mathbb{Q} \circ X^{-1}$ the image measure of $\mathbb{Q}$ under $X$.
  And, for any non-negative measurable function $f$ on $\Omega$, not necessarily normalized, we denote by $f\mathbb{P}$ the associated (possibly non-normalized) measure $\mathbb{Q}$, defined by
\[
\mathbb{Q}(A) = \int_A f \, \mathrm{d}\mathbb{P}, \qquad A \in \mathcal{F}.
\]

\paragraph{Duality.} By definition of $\mathcal{S}$ and $\mathcal{S}^\star$ in 
\eqref{eq:S}
and \eqref{eq:expo:bound:psi}, 
we have for any $\mathbb{F}$-progressively measurable processes $q$ and $\zeta$, valued in $\mathbb{R}$,  
\begin{equation} \label{ineq:S-S-star}
    \mathcal{S}(q) + \mathcal{S}^\star(\zeta) \geq \frac{1}{\gamma} \mathbb{E}\left[\int_0^T q_s |\zeta_s|^2 \dd s \right],
\end{equation}
where $\mathcal{S}(q)$ and $\mathcal{S}^\star(\zeta)$ might take infinite values. Equality holds whenever 
\begin{equation*}
    q \in \argmax_{q' \in \mathcal{Q}} \left\{\mathbb{E}\left[ \int_0^T q_s |\zeta_s|^2 \dd s \right] - \gamma \mathcal{S}(q)\right\}.
\end{equation*}

\paragraph{Miscellaneous.} 
For finite-dimensional vectors $x$ and $y$ (in the same space), $x \cdot y$ denotes their scalar product. 
We also define the entropy function $\textrm{\rm Ent} \colon \mathbb{R}_+ \to \mathbb{R}$:
\begin{equation}
\label{eq:entropy:definition:h}
    \textrm{\rm Ent}(x) := x(\ln(x) - 1).
\end{equation}


\section{Robust control within a fixed environment} 
\label{sec:robust}

In this section, we address the problem \eqref{pb:optim-mfg}
 when $\mu \in {\mathcal M}_{2-r}({\mathbb R}^n)$ is fixed. 
 The set-up is clarified in  Subsection
\ref{subsec:set-up}. In Subsection 
\ref{subsec:solvability}, we present an  existence and uniqueness result to \eqref{pb:optim-mfg}, 
which directly follows from our companion work 
 \cite{DelarueLavigne}.
In Subsection 
\ref{subse:3:3}, we derive stability estimates, which are key 
in the analysis of the mean-field game carried out in the next section.

\subsection{Set-up}
\label{subsec:set-up}

The assumptions are mostly derived from the analysis introduced in 
\cite{DelarueLavigne}. 

We use repeatedly the notion of \textit{progressive-measurable} field.
For a metric space $({\mathcal X},d)$
and an integer $k \in {\mathbb N}^*$, a random field 
$\mathcal{F} : \Omega \times [0,T] \times {\mathcal X} \rightarrow {\mathbb R}^k$ is  progressively-measurable if, for any $t \in [0,T]$,
its restriction
to $\Omega \times [0,t] \times {\mathcal X}$
is ${\mathcal F}_t \otimes {\mathcal B}([0,t]) \otimes {\mathcal B}({\mathcal X})/{\mathcal B}({\mathbb R}^k)$
measurable.

Throughout, $L$ and $r$ are two constants, with $L>0$ and $r \in \{0,1\}$.
The assumptions hold true for any fixed 
$\mu \in {\mathcal M}_{2-r}({\mathbb R}^n)$, 
and the constants $L$ and $r$ are assumed to be independent of $\mu$. 
In fact, the only assumption in which $\mu$ appears is \ref{hyp:g:mu:fixed}. 
Therein, we pay special attention to 
introduce a tailored notation for the constants that genuinely depend on $\mu$.
    
\begin{enumerate}[label*=A\arabic*]
    \item \label{hyp:b} \textit{Initial condition and drift.}
The initial condition 
$\eta$ belongs to $L^\infty({\mathcal F}_0,{\mathbb R}^n)$, i.e. $\|\eta \|_{L^\infty(\mathcal{F}_0,\mathbb{R}^{n})}  <+\infty$, and the drift $b \colon \Omega \times [0,T] \times \mathbb{R}^n \times \mathbb{R}^n \to \mathbb{R}^n$ is linear, i.e., 
\begin{align*}
    b(t,x,\psi) = a_t + b_t x + c_t \psi,
\end{align*}
with $a$, $b$ and $c$ in 
$L^\infty({\mathbb F},{\mathbb R}^n)$,
$L^\infty({\mathbb F},{\mathbb R}^{n \times n})$
and
$L^\infty({\mathbb F},{\mathbb R}^{n \times n})$,  and 
 $\|a\|_{L^\infty(\mathbb{F},\mathbb{R}^{n})} + \|b\|_{L^\infty(\mathbb{F},\mathbb{R}^{n \times n})} +  \|c\|_{L^\infty(\mathbb{F},\mathbb{R}^{n \times n})}  \leq L$.
 
 \item \textit{Volatility.} \label{hyp:sigma} The volatility $\sigma \colon \Omega \times [0,T] \times \mathbb{R}^n \to \mathbb{R}^{n \times d}$ is linear, the $n \times d$ entries of the matrix $\sigma$ being of the form
\begin{align*}
    (\sigma(t,\psi))_{i,j} = (\nu_t)_{i,j} + r (\sigma_t)_{i,j,k} \psi_k,
\end{align*}
for $(i,j,k) \in \{1,\ldots,n\} \times \{1,\ldots,d\} \times  \{1,\ldots,n\}$. Above, 
$\nu$ and $\sigma$ are in $L^\infty(\mathbb{F},\mathbb{R}^{n\times d})$
and
$L^\infty(\mathbb{F},\mathbb{R}^{n\times d})$, 
and 
$\|\nu\|_{L^\infty(\mathbb{F},\mathbb{R}^{n\times d})} + \|\sigma\|_{L^\infty(\mathbb{F},\mathbb{R}^{n\times d \times n})} \leq L$.

\item \textit{Driver.} \label{eq:assumption1-f}  The driver $f \colon \Omega \times [0,T] \times \mathbb{R} \times \mathbb{R}^d \to \mathbb{R}$ is progressively-measurable, and
twice continuously differentiable and 
convex in its last two last arguments. 
There exist two positive constants $\alpha,\beta$ 
such that, almost surely in $\omega$ and almost everywhere in 
$t$,
\begin{equation*}
        f(t,y,z) \leq |f_t^0| + \alpha |y| + \frac{\beta}{2} |z|^2,
    \quad (y,z) \in {\mathbb R} \times {\mathbb R}^d,
\end{equation*}
where  $f^0 \coloneqq f(0,0) \in L^\infty(\mathbb{F})$. 
The second order derivatives in $y$ and $z$ are bounded by $L$. 

\item \textit{Dual driver.} \label{eq:assumption1-nabla-f-star} 
We call $f^\star \colon \Omega \times [0,T] \times \mathbb{R} \times \mathbb{R}^d  \to \mathbb{R}$ the Fenchel transform of the driver $f$ with respect to its variables $(y,z)$,
\begin{equation*}
\label{eq:f:star}
    f^\star (t,y^\star,z^\star) := \sup_{(y,z) \in \mathbb{R}\times \mathbb{R}^d} \left \{ \langle (y^\star,z^\star),(y,z) \rangle - f(t,y,z) \right\}. 
\end{equation*}
It is shown in \cite{DelarueLavigne} that $f^\star$
is progressively-measurable. 

\item \textit{Running cost.}  \label{hyp:L}
The cost $\ell \colon \Omega \times [0,T] \times \mathbb{R}^n \to \mathbb{R}$ is progressively-measurable and strongly convex, twice differentiable, and has a quadratic growth with respect to the control variable:
\begin{equation*} 
    \bigl(\nabla_\psi \ell(t,\psi) - \nabla_\psi \ell(t,\psi'\bigr))\cdot (\psi - \psi') \geq \frac{1}{L}|\psi - \psi'|^2, \quad |\nabla^2_{\psi} \ell(t,\psi) | \leq L,
\end{equation*}
and $\vert \ell(t,0) \vert \leq L$ for any $t\in[0,T]$ and $\psi,\psi' \in \mathbb{R}^n$.

\item \textit{Coefficients.} \label{hyp:gamma} We fix the coefficient $\gamma$ in 
\eqref{eq:expo:bound:psi}
to be given by
\begin{equation*}
\begin{split}
    \gamma &= 8 \beta \max(1, L) e^{\alpha T} \|\Gamma\|_{L^\infty({\mathbb F},{\mathbb R}^{n \times n})}
    \|\Gamma^{-1}\|_{L^\infty({\mathbb F},{\mathbb R}^{n \times n})}
    \\
    &\hspace{15pt} \times \left(\|\nu\|_{L^\infty({\mathbb F},{\mathbb R}^{n \times d})} + 12 \max(1,L) e^{\alpha T} \|\sigma\|_{L^\infty({\mathbb F},{\mathbb R}^{n \times d \times n})} \right).
    \end{split}
\end{equation*} 
When $r = 0$ we assume that the coefficients satisfy the condition 
\begin{equation*}
    4 \beta e^{\alpha T} L \|\Gamma \|^2_{L^\infty(\mathbb{F},\mathbb{R}^{n \times n})}
    \|\Gamma^{-1} \|^2_{L^\infty(\mathbb{F},\mathbb{R}^{n \times n})}
    \|\nu\|^2_{L^\infty(\mathbb{F},\mathbb{R}^{n \times d})} T < 1,
\end{equation*}
where $\Gamma$ is the solution to 
\begin{equation*}
        \frac{\dd}{\dd t} \Gamma_t =  b_t \Gamma_t, \quad 
        t \in [0,T],
        \quad 
        \Gamma_0 = I_n,
    \end{equation*}
    with $I_n$ standing for the $n \times n$ identity matrix. 

\item \textit{Terminal cost.} \label{hyp:g:mu:fixed}
    We  assume that $g : {\mathbb R}^n \times {\mathcal M}_{2-r}({\mathbb R}^n) \rightarrow {\mathbb R}$ is convex and twice differentiable in $x$ and, for any real  
    $C \geq 0$, there exists a constant $L_C \geq 0$ such that, for any 
    $\mu \in {\mathcal M}_{2-r}({\mathbb R}^n)$ with 
    $M_{2-r}(\mu) \leq C$, 
            \begin{equation} \label{hyp:G-growth-JC}
            \begin{array}{rl}
                    -L _C
        \left(1
        +  |x|\right) \leq g(x,\mu)  &\leq L_C \left(1 +    |x|^{2-r}\right), \\[.5em]
        \left |\nabla_x g(x,\mu) \right| & \leq L_C\left(1         +  |x|^{1-r} \right),\\[.5em]
           |\nabla^2_x g(x,\mu)| &\leq L_C.
            \end{array}
        \end{equation}
\end{enumerate}

\begin{remark}
The growth condition \ref{eq:assumption1-f}
implies that $f^\star(t,y^\star,z^\star) = + \infty$ if $\vert y^\star \vert > \alpha$. In particular, 
$Y^\star$ in \eqref{eq:q} is necessarily bounded by $\alpha$ if ${\mathcal S}(q)$ is finite (as it is required). Moreover, it is proven in
the first step of the proof of 
\cite[Proposition 20]{DelarueLavigne} that there exists a constant $C$, only depending on the parameters in the standing assumption such that, for 
$q \in {\mathbb Q}$, 
\begin{equation}
\label{eq:H:S}
{\mathbb E}[\textrm{\rm Ent}(q)] \leq C(1+ {\mathcal S}(q)), 
\end{equation}
where $\textrm{\rm Ent}$ is given by 
\eqref{eq:entropy:definition:h}.

In \ref{hyp:gamma}, the choice of $\gamma$ is dictated by the analysis carried out in 
\cite{DelarueLavigne}. 

In \ref{hyp:g:mu:fixed}, $g$ is assumed to be deterministic (contrary to the other coefficients). 
In fact, $g$ could be allowed to be random in some 
of the statements, but it is typically deterministic in the whole discussion on uniqueness 
and on the $N$-player approximation.

\end{remark}

\subsection{Solvability of the robust control problem}
\label{subsec:solvability}

Following \cite{DelarueLavigne}, we study the optimization problem \eqref{pb:optim-mfg} (for a fixed $\mu \in {\mathcal M}_{2-r}({\mathbb R}^n)$) via the associated Pontryagin system. Under the standing assumptions, the cost functionals are convex with respect to the state variable $X$ and concave with respect to the control variable $q$. As a consequence, the Pontryagin principle yields a full characterization of the saddle point. This constitutes one of the main results of \cite{DelarueLavigne}.

As already explained in  Introduction, the Pontryagin system takes the form of two   FBSDEs, each backward equation being driven by the derivative (with respect to the corresponding coordinate) of the corresponding pre-Hamiltonian introduced in \eqref{eq:def:H:F}.

Given $q \in \mathcal{Q}$, the first order condition \eqref{optim:condition-primal} for the representative player problem
writes in the form of a foward-backward system, 
with $(\psi,p,k,X)$ as unknown:
\begin{equation} \label{optim:condition-primal} \tag{RP$[\mu]$}
    \left\{ \begin{array}{rll}
        - \dd p_t  & =   \nabla_x H(t,X_t,\psi_t,p_t,k_t,q_t) \dd t - k_t \dd W_t 
        \\
        &= b_t^\top p_t \dd t - k_t \dd W_t,        
        & 
        p_T = q_T \nabla_x g(X_T^\psi,\mu), \\[0.8em]
        \dd X_t & = b(t,X_t,\psi_t) \dd t + \sigma(t,\psi_t) \dd W_t,  & X_0 = \eta,\\[0.8em]
        \psi_t & \in \argmin_{\alpha} H(t,X_t,\alpha,p_t,k_t,q_t), 
        \\
\textrm{\rm i.e.} \quad         
        0 & = q_t \nabla_{\psi} \ell(t,\psi_t) +   p_t \cdot c_t + r \Tr (\sigma_t^\top k_t),
        & \dd \mathbb{P}  \otimes \dd t \text{-a.s.},
    \end{array} \right.
\end{equation}
where we denote by convention
\begin{equation} \label{def:trace-sigma-k}
    {\rm Tr}\left( \sigma_t^{\top} k_t \right) =
\left( 
 \sum_{i=1}^n
 \sum_{j=1}^d
( \sigma_t)_{i,j,\ell} (k_t)_{i,j} 
\right)_{\ell=1,\ldots,d}.
\end{equation}
Solutions $(\psi,p,k,X)$ to 
\eqref{optim:condition-primal} are sought within the space
\begin{equation}
\label{def:mathscr-A}
     \mathscr{A}  \coloneqq \mathcal{A} \times D({\mathbb F})
    \times (
    \underset{\beta \in (0,1)}{\cap} M^{\beta}({\mathbb F},{\mathbb R}^d)) \times S^{2-r}(\mathbb F,\mathbb{R}^n,\mathbb{Q}), 
\end{equation}
where $\mathbb{Q}$ in the first line is the equivalent  measure associated to $q$, i.e., ${\mathbb Q}= q {\mathbb P}$. 
\vskip 5pt 

\noindent Given $\psi \in \mathcal{A}$, the first order condition \eqref{optim:condition-dual} for the nature problem
writes in the form of another forward-backward system, 
with 
 $(Y,Z,q)$ as unknown: 
\begin{equation} \label{optim:condition-dual} \tag{N$[\mu]$}
    \left\{ \begin{array}{rll}
        - \dd Y_t &=  \partial_q F(t,q_t,Y^{\star}_t,Z^{\star}_t,Y_t,Z_t,\psi_t) \dd t - Z_t \cdot \dd W_t        \\
&= (f(t,Y_t,Z_t) + \ell(t,\psi_t))\dd t - Z_t \cdot \dd W_t,  & Y_T =g(X_T^\psi,\mu),         
        \\[0.8em]
         \dd q_t &  = q_t Y^\star_t  \dd t +  q_t Z^\star_t  \cdot \dd W_t, &
        q_0  = 1,\\[0.8em]
        (Y^\star_t,Z^\star_t) & \in \underset{{(Y^{\star \prime},Z^{\star \prime})}}{\argmax}F(t,q_t,Y^{\star \prime},Z^{\star \prime},Y_t,Z_t,\psi_t)  \\
        \Leftrightarrow  
        (Y^\star_t,Z^\star_t) &= (\partial_y f(t,Y_t,Z_t),\nabla_z f(t,Y_t,Z_t)), 
        & \dd \mathbb{P}  \otimes \dd t\text{-a.s.}
    \end{array} \right.
\end{equation}
Solutions 
$(q,Y,Z)$ to \eqref{optim:condition-dual}
are sought in the space
\begin{equation} 
\label{def:mathscr-Q}
    \mathscr{Q}  \coloneqq    {\mathcal Q}\times  D({\mathbb F},{\mathbb Q})
    \times (
    \underset{\beta \in (0,1)}{\cap} M^{\beta}({\mathbb F},{\mathbb R}^d,{\mathbb Q})    
    ).
\end{equation}
\vskip 5pt

Here is now the main statement 
of \cite{DelarueLavigne}
regarding
the inf-sup mean-field stochastic control problem
\eqref{pb:optim-mfg}.


\begin{theorem} \label{theorem:SMP}
Let $\mu \in {\mathcal M}_{2-r}({\mathbb R}^n)$. 
     Then, there exists 
    a unique saddle point
    $({\psi},{q}) \in \mathcal{A} \times \mathcal{Q}$ to Problem \eqref{pb:optim-mfg}, i.e. 
        \begin{equation*}
            \min_{\psi' \in \mathcal{A}} \max_{q' \in \mathcal{Q}} \mathcal{J}(q',\psi') = \max_{q' \in \mathcal{Q}} \min_{\psi' \in \mathcal{A}}  \mathcal{J}(q',\psi') = \mathcal{J}(q,{\psi}).
        \end{equation*}
Moreover, if a pair  
$(\psi,q) \in \mathcal{A} \times \mathcal{Q}$ is a solution to the problem \eqref{pb:optim-mfg}, then
the tuples $(\psi,p,k,X)$, obtained by solving in ${\mathscr A}$ the two decoupled equations in 
\eqref{optim:condition-primal},
and $(q,Y,Z)$, obtained by solving in ${\mathscr Q}$ the two decoupled equations in 
\eqref{optim:condition-dual},
satisfy the optimality conditions in 
\eqref{optim:condition-primal}
and \eqref{optim:condition-dual}
respectively. Conversely,
if
$(\psi,p,k,X,q,Y,Z) \in \mathscr{A} \times  \mathscr{Q} $ is the solution to \eqref{optim:condition-primal}-\eqref{optim:condition-dual}, then the pair 
$(q,\psi) \in \mathcal{Q} \times \mathcal{A}$ is a solution to the problem \eqref{pb:optim-mfg}.
\end{theorem}

In the rest of the article, we denote the 
unique saddle point by $(q^\mu,\psi^\mu)$. Accordingly, the solution to 
\eqref{optim:condition-primal}
is denoted by $(\psi^\mu,p^\mu,k^\mu,X^\mu)$ and 
the solution to 
\eqref{optim:condition-dual}
is denoted by $(q^\mu,Y^\mu,Z^\mu)$. 
The representatives $(Y^\star,Z^\star)$ 
of $q^\mu$ in \eqref{eq:q} are denoted by 
$(Y^{\star,\mu},Z^{\star,\mu})$.

The fact that the cost 
${\mathbb E}[q_T^\mu g(X_T^{\mu}, \mu)]$ is well-defined is the consequence of \ref{hyp:g:mu:fixed}
and of the 
fact that $X^\mu \in S^{2-r}({\mathbb F},{\mathbb R}^n,{\mathbb Q})$. 
Generally speaking, the latter is a consequence of the 
following lemma, which corresponds to \cite[Lemma 41]{DelarueLavigne}:

\begin{lemma} \label{lemma:reg-X-psi-A} 
Let $(q,\psi) \in {\mathcal Q} \times {\mathcal A}$. Then, 
$X^\psi$ belongs to $S^{2-r}(\mathbb{F},\mathbb{Q},\mathbb{R}^n)$, where $r \in \{0,1\}$ is as in 
\ref{hyp:sigma}, and
there exists a constant $C$, independent of $q$ and $\psi$, such that 
\begin{equation*}
     \mathbb{E}\left[ q_T \left|\sup_{t \in [0,T]} \vert X_t^\psi\vert \right|^{2-r} \right] \leq C \left( 1 + \mathcal{S}(q) + \mathcal{S}^\star(\psi)\right).
\end{equation*}
\end{lemma}

\subsection{Stability Estimates}
\label{subse:3:3}

The next result is taken
from \cite{DelarueLavigne}. It is in fact 
part of the proof on which the derivation of  the stochastic minimum principle for 
the problem \eqref{optim:condition-primal} relies. 

\begin{lemma} \label{lemma:sufficient}
Let $\mu, \tilde{\mu} \in \mathcal{M}_{2-r}(\mathbb{R}^n)$. 
Then,
    \begin{equation} \label{ineq:p-delta-X}
        \mathbb{E}\left[p_T^\mu \cdot \left(X_T^{\tilde \mu}- X_T^\mu\right) \right] = - \mathbb{E}\left[ \int_0^T 
        {q}_s^\mu \nabla_{\psi} \ell\left(s,\psi_s^\mu\right) \cdot \left(\psi_s^{\tilde \mu} - \psi_s^\mu)\right) \dd s\right],
    \end{equation}
    which implicitly implies that the expectations right above are well-defined. 
\end{lemma}

\begin{proof}
This is the penultimate display in the proof of \cite[Lemma 38]{DelarueLavigne}.
\end{proof}

We now recall the following result from convex analysis. 

\begin{lemma}
\label{lem:c:convex}
There exists a constant $c>0$, only depending on the parameters in 
\ref{hyp:b}-\ref{hyp:g:mu:fixed}, 
 such that, for any 
 $t \in [0,T]$, 
 $y_1^\star,y_2^\star \in {\mathbb R}$, 
 $z_1^\star,z_2^\star \in {\mathbb R}^d$ and 
 $\theta \in (0,1)$,
 \begin{equation}
 \label{eq:strict:convexity:1}
     \begin{split}
         f^\star\left(t,\theta
         y_1^\star + (1-\theta) 
         y_2^\star, \theta 
         z_1^\star + (1- \theta) 
         z_2^\star \right) 
         &\leq 
         \theta f^\star \left( t, 
         y_1^\star,z_1^\star \right) 
         + (1- \theta) 
         f^\star \left( t, 
         y_2^\star,z_2^\star 
         \right) 
         \\
&\hspace{15pt}         - c \theta ( 1- \theta)
         \left( 
         \vert 
         y_1^\star - y_2^\star \vert^2 + 
         \vert z_1^\star - z_2^\star 
         \vert^2 
         \right). 
              \end{split}
 \end{equation}
 In particular, 
 for any $t \in [0,T]$,
$y \in {\mathbb R}$, $z \in {\mathbb R}^d$, $h_y \in {\mathbb R}$  and $h_z \in {\mathbb R}^d$,  
\begin{equation}
 \label{eq:strict:convexity:2}
\begin{split}
&f^\star\left(t,\partial_yf(t,y,z) + h_y,\partial_z f(t,y,z)+h_z\right)
\\
&\geq f^\star(\partial_yf(t,y,z) ,\partial_z f(t,y,z))
  + y\cdot h_y + z\cdot h_z + c \left( \vert h_y\vert^2 +\vert h_z \vert^2\right).
\end{split}
\end{equation}
\end{lemma}

\begin{proof}
The proof is divided in two steps. 
\vskip 4pt

\noindent \textit{Step 1.}
We start with the following preliminary step. For 
$y^\star \in {\mathbb R}$
and $z^\star \in {\mathbb R}^d$ satisfying 
$f^\star(t,y^\star,z^\star) < + \infty$, there exists a sequence 
$(y_n,z_n)_{n \in {\mathbb N}^*}$ in 
${\mathbb R} \times {\mathbb R}^d$ such that 
\begin{equation}
\label{eq:strict:convexity:3}  
f^\star(t,y^\star,z^\star) \leq y^\star y_n + z^\star \cdot z_n - f(t,y_n,z_n) + \frac1n, 
\quad n \in {\mathbb N}^*. 
\end{equation}
And then,
for $h_y,k_y\in {\mathbb R}$
and $h_z,k_z \in {\mathbb R}^d$, 
\begin{equation*} 
\begin{split} 
&f^\star(t,y^\star+h_y,z^\star+h_z)
\\
&\geq 
(y_n + k_y) (y^\star + h_y) 
+ (z_n + k_z) \cdot (z^\star + 
h_z) - f(t,y_n+k_y,z_n+k_z) 
\\
&=  
y_n y^\star 
+ z_n \cdot z^\star 
- f(t,y_n+k_y,z_n+k_z)
\\
&\hspace{15pt} 
+ 
y_n h_y + z_n \cdot h_z 
+ 
k_y y^\star + 
k_z \cdot z^\star 
\\
&\hspace{15pt} 
+   h_y k_y + h_z \cdot k_z
\\
&\geq 
f^\star(t,y^\star,z^\star) 
+ f(t,y_n,z_n) 
- f(t,y_n+k_y,z_n+k_z)
- \frac1n
\\
&\hspace{15pt} 
+ 
y_n h_y + z_n \cdot h_z 
+ 
k_y y^\star + 
k_z \cdot z^\star  
\\
&\hspace{15pt} 
+   h_y k_y + h_z \cdot k_z.
\end{split}
\end{equation*}
Using the regularity properties of
$f$ stated in 
\ref{eq:assumption1-f}, 
we deduce that there exists a constant $C>0$ such that 
\begin{equation*}
    \begin{split}
&f^\star(t,y^\star+h_y,z^\star+h_z)
\\
&\geq f^\star(t,y^\star,z^\star) - \partial_y f(t,y_n,z_n)  k_y 
- \partial_z f(t,y_n,z_n) \cdot k_z 
- 
C \left( \vert k_y \vert^2+ 
\vert k_z \vert^2 \right) 
- 
\frac1n
\\
&\hspace{15pt} 
+ 
y_n h_y + z_n \cdot h_z 
+ 
k_y y^\star + 
k_z \cdot z^\star  
\\
&\hspace{15pt} 
+   h_y k_y + h_z \cdot k_z. 
\end{split}
\end{equation*} 
Choose now 
$k_y=h_y/(2C)$ and 
$k_z=h_z/(2C)$, and deduce 
\begin{equation}
\label{eq:strict:convexity:4}
    \begin{split}
f^\star(t,y^\star+h_y,z^\star+h_z)
&\geq f^\star(t,y^\star,z^\star) 
+ 
\frac1{4 C} \left( \vert h_y \vert^2+ 
\vert h_z \vert^2 \right) 
- 
\frac1n
\\
&\hspace{15pt} -
\frac1{2C} \partial_y f(t,y_n,z_n)  h_y 
- \frac1{2C} \partial_z f(t,y_n,z_n) \cdot h_z 
\\
&\hspace{15pt} 
+ 
y_n h_y + z_n \cdot h_z 
+ 
\frac1{2C} h_y y^\star + 
\frac1{2C}
h_z \cdot z^\star.
\end{split}
\end{equation} 
\vskip 4pt

\noindent \textit{Step 2.}
We now derive \eqref{eq:strict:convexity:1}. 
Apply \eqref{eq:strict:convexity:4}  twice, once with $\theta(h_y,h_z)$ substituted for $(h_y,h_z)$ and once with 
$-(1-\theta) (h_y,h_z)$ substituted for 
$(h_y,h_z)$, and make the 
$(1-\theta,\theta)$ convex combination of the resulting two inequalities. By linearity of the terms on the last line in \eqref{eq:strict:convexity:4}, we get
\begin{equation*} 
\begin{split} 
&\theta f^\star\left(t,y^\star -(1- \theta) h_y,z^\star - (1- \theta) 
h_z 
\right)
+
(1- \theta) 
f^\star(t,y^\star + \theta h_y, 
z^\star + \theta h_ z)
\\
& \geq  f^\star(t,y^\star,z^\star)  
   + \frac{1}{4C} \theta (1-\theta)
   \bigl( \vert h_y \vert^2 + \vert h_z \vert^2 \bigr) - \frac{1}n. 
\end{split}
\end{equation*} 
Letting  $n$ to $+\infty$ and, 
for $y_1^\star$, $y_2^*$,
$z_1^\star$ and $z_2^\star$
as in the statement such that 
$f^\star(t,y_1^\star,z_1^\star)$
and $f^\star(t,y_2^\star,z_2^\star)$
are finite, 
we apply the above display with 
$y^\star = \theta y_1^\star + (1- \theta) y_2^\star$, 
$z^\star = \theta z_1^\star + (1-\theta) z_2^\star$ (by convexity, 
$f^\star(t,y^\star,z^\star) < + \infty$), 
$h_y= y_2^\star-y_1^\star$
and
$h_z = z_2^\star - z_1^\star$
(so that 
$y^\star - (1-\theta) h_y 
=y_1^\star$ and
$y^\star + \theta h_y 
=y_2^\star$, 
and similarly when $z$ is substituted for $y$). We get 
\eqref{eq:strict:convexity:1}
(when $f^\star(t,y_1^\star,z_1^\star)$ and $f^\star(t,y_2^\star,z_2^\star)$ are finite). When 
$f^\star(t,y_1^\star,z_1^\star)$ or $f^\star(t,y_2^\star,z_2^\star)$ 
is $ + \infty$, \eqref{eq:strict:convexity:1} is necessarily true. 

\vskip 4pt

\noindent \textit{Step 3.}
In order to get \eqref{eq:strict:convexity:2}, we recall from Fenchel's duality  that, for any $y \in {\mathbb R}$ and 
$z \in {\mathbb R}^d$, 
\begin{equation*} 
\begin{split}
&f^\star\left(t,\partial_y f(t,y,z),\partial_z f(t,y,z) 
\right) + f(t,y,z) = y   \partial_y f(t,y,z)
+ z \cdot 
\partial_z f(t,y,z). 
\end{split}
\end{equation*} 
This gives an identity in \eqref{eq:strict:convexity:3},
when $y^\star = \partial_y f(t,y,z)$, $z^\star= \partial_z f(t,y,z)$, 
$y_n=y$, $z_n=z$, 
and $1/n$ is formally replaced by $0$. 
Then, \eqref{eq:strict:convexity:4}
implies 
\eqref{eq:strict:convexity:2}. 
This completes the proof. 
\end{proof}

Here is now the main result of this subsection: 
\begin{proposition}
\label{prop:main:estimate}
There exists a constant $c>0$, only depending on the parameters in 
\ref{hyp:b}-\ref{hyp:g:mu:fixed}, such that, 
the following two inequalities hold true, for 
any 
 $\mu,\tilde \mu \in {\mathcal M}_{2-r}({\mathbb R}^n)$:
\begin{equation}
\label{eq:corol:9:first:claim} 
\begin{split}
&{\mathbb E} \left[ q_T^{\tilde \mu} \left( g(X_T^\mu,\tilde \mu) 
-
g(X_T^\mu,  \mu) \right) 
 + q_T^{ \mu} \left( g(X_T^{\tilde \mu},  \mu) 
-
g(X_T^{\tilde \mu},\tilde \mu) \right) 
\right]  
\\
&\geq c 
{\mathbb E} \left[ \int_0^T
\left( q_t^{ \mu} 
+q_t^{\tilde \mu} 
\right) 
\left(\vert \psi_t^{\tilde \mu} - \psi_t^{\mu} 
\vert^2 + \vert Y_t^{\star,\tilde \mu} - Y_t^{\star,\mu} \vert^2 
+ \vert Z_t^{\star,\tilde \mu} - Z_t^{\star,\mu} \vert^2 \right) \dd t \right],
\end{split}
\end{equation}
and
\begin{equation}
\label{corol:9:step2:end} 
\begin{split}
&{\mathbb E} \left[ \left( q_T^\mu 
- q_T^{\tilde \mu} \right) \left( g(X_T^\mu,\mu) - g(X_T^{\tilde \mu},\tilde \mu) \right)
\right]
\\
&\geq
{\mathbb E} \left[  \left( q_T^{\mu} \nabla_x g(X_T^{\mu}, \mu) 
- q_T^{\tilde \mu} \nabla_x g(X_T^{\tilde \mu},\tilde \mu) \right) 
\cdot (X_T^{\mu} - X_T^{\tilde \mu})
\right]
\\
&\hspace{15pt} + c 
{\mathbb E} \left[ \int_0^T
\left( q_t^{\mu} + q_t^{\tilde \mu}\right)  \left(\vert \psi_t^{\tilde \mu} - \psi_t^{\mu} 
\vert^2 + \vert Y_t^{\star,\tilde \mu} - Y_t^{\star,\mu} \vert^2 
+ \vert Z_t^{\star,\tilde \mu} - Z_t^{\star,\mu} \vert^2 \right) \dd t \right].
\end{split}
\end{equation}
\end{proposition}

The fact that the expectations on the first line of \eqref{eq:corol:9:first:claim}
and on the first and second lines of \eqref{corol:9:step2:end} are well-defined is a consequence of Lemma \ref{lemma:reg-X-psi-A}.

\begin{proof}
The proof is divided in three steps. 
\vskip 4pt

\noindent \textit{Step 1.}
From  the proof of  \cite[Lemma 30]{DelarueLavigne} (starting from the penultimate display in the proof, and then using (189) and (187) therein), 
we have
\begin{equation*} 
\begin{split}
&{\mathbb E} \left[ q_T^\mu g(X_T^\mu,\mu) 
+ \int_0^T q_t^\mu \ell(t,\psi_t^\mu) \dd t 
\right] - {\mathcal S}(q^\mu)
\\
&\geq
{\mathbb E} \left[ q_T^{\tilde \mu} g(X_T^\mu,\mu) 
+ \int_0^T q_t^{\tilde \mu} \ell(t,\psi_t^\mu) \dd t 
\right] - {\mathcal S}(q^{\tilde \mu})
  + \lim_{A \rightarrow \infty}
{\mathbb E} 
\left[ \int_0^{T \wedge \tau_A} 
q_t^{\tilde \mu} \Delta f_t^\star \dd t \right],
\end{split}
\end{equation*}
where 
$(\tau_A)_{A >0}$ is a collection of stopping time that converges almost surely to $T$ (as $A$ tends to $+\infty$), and 
\begin{equation*}
\begin{split}
 \Delta f_t^\star
&= 
f^\star\left(t,Y_t^{\star,\tilde \mu},Z_t^{\star,\tilde \mu}\right) 
-
f^\star\left(t,Y_t^{\star,  \mu},Z_t^{\star,  \mu}\right) 
\\
&\hspace{15pt} - 
\left( Y_t^{\star,\tilde \mu} - Y_t^{\star,  \mu} \right) Y_t^\mu 
- 
\left( Z_t^{\star,\tilde \mu} - Z_t^{\star,  \mu} \right) \cdot Z_t^\mu.
\end{split}
\end{equation*} 
By Lemma  
\ref{lem:c:convex}, we obtain 
\begin{equation*} 
\begin{split}
&{\mathbb E} \left[ q_T^\mu g(X_T^\mu,\mu) 
+ \int_0^T q_t^\mu \ell(t,\psi_t^\mu) \dd t 
\right] - {\mathcal S}(q^\mu)
\\
&\geq
{\mathbb E} \left[ q_T^{\tilde \mu} g(X_T^\mu,\mu) 
+ \int_0^T q_t^{\tilde \mu} \ell(t,\psi_t^\mu) \dd t 
\right] - {\mathcal S}(q^{\tilde \mu})
\\
&\hspace{15pt} + c 
{\mathbb E} \left[ \int_0^T
q_t^{\tilde \mu} \left( \vert Y_t^{\star,\tilde \mu} - Y_t^{\star,\mu} \vert^2 
+ \vert Z_t^{\star,\tilde \mu} - Z_t^{\star,\mu} \vert^2 \right) \dd t \right].  
\end{split}
\end{equation*}
And then, 
by strong convexity of $\ell$ (see \ref{hyp:L}), we obtain (for a new value of $c$), 
\begin{equation*} 
\begin{split}
&{\mathbb E} \left[ q_T^\mu g(X_T^\mu,\mu) 
+ \int_0^T q_t^\mu \ell(t,\psi_t^\mu) \dd t 
\right] - {\mathcal S}(q^\mu)
\\
&\geq
{\mathbb E} \left[ q_T^{\tilde \mu} g(X_T^\mu,\mu) 
+ \int_0^T q_t^{\tilde \mu} \ell(t,\psi_t^{\tilde \mu}) \dd t 
+ 
\int_0^t q_t^{\tilde \mu} \nabla_\psi \ell(t,\psi_t^{\tilde \mu})
\cdot \left( \psi_t^{\mu} - \psi_t^{\tilde \mu} \right) 
\dd t
\right] - {\mathcal S}(q^{\tilde \mu})
\\
&\hspace{15pt} + c 
{\mathbb E} \left[ \int_0^T
q_t^{\tilde \mu} \left( 
\vert \psi_t^{\tilde \mu} - \psi_t^{\mu} 
\vert^2 + 
\vert Y_t^{\star,\tilde \mu} - Y_t^{\star,\mu} \vert^2 
+ \vert Z_t^{\star,\tilde \mu} - Z_t^{\star,\mu} \vert^2 \right) \dd t \right].  
\end{split}
\end{equation*}
By Lemma \ref{lemma:sufficient}, 
\begin{equation}
\label{corol:9:step1:end} 
\begin{split}
&{\mathbb E} \left[ q_T^\mu g(X_T^\mu,\mu) 
+ \int_0^T q_t^\mu \ell(t,\psi_t^\mu) \dd t 
\right] - {\mathcal S}(q^\mu)
\\
&\geq
{\mathbb E} \left[ q_T^{\tilde \mu} g(X_T^\mu,\mu) 
- q_T^{\tilde \mu} \nabla_x g(X_T^{\tilde \mu},\tilde \mu) 
\cdot (X_T^{\mu} - X_T^{\tilde \mu})
+ \int_0^T q_t^{\tilde \mu} \ell(t,\psi_t^{\tilde \mu}) \dd t 
\right] - {\mathcal S}(q^{\tilde \mu})
\\
&\hspace{15pt} + c 
{\mathbb E} \left[ \int_0^T
q_t^{\tilde \mu} \left(\vert \psi_t^{\tilde \mu} - \psi_t^{\mu} 
\vert^2 + \vert Y_t^{\star,\tilde \mu} - Y_t^{\star,\mu} \vert^2 
+ \vert Z_t^{\star,\tilde \mu} - Z_t^{\star,\mu} \vert^2 \right) \dd t \right].
\end{split}
\end{equation}

\noindent 
\textit{Step 2.} 
We derive the first claim. 
By convexity of $g$ in the first variable, 
we get the following bound for the first term on the second line of 
\eqref{corol:9:step1:end}
\begin{equation*}
g(X_T^\mu,\tilde \mu) \geq 
g(X_T^{\tilde \mu},\tilde \mu)
+ 
 \nabla_x g(X_T^{\tilde \mu},\tilde \mu) \cdot (X_T^\mu - X_T^{\tilde \mu}),
\end{equation*}
from which we deduce that 
\begin{equation*} 
\begin{split}
&{\mathbb E} \left[ q_T^\mu g(X_T^\mu,\mu) 
+ \int_0^T q_t^\mu \ell(t,\psi_t^\mu) \dd t 
\right] - {\mathcal S}(q^\mu)
\\
&\geq
{\mathbb E} \left[ q_T^{\tilde \mu} \left( g(X_T^\mu,\mu) 
-
g(X_T^\mu,\tilde \mu)
+
g(X_T^{\tilde \mu},\tilde \mu)
 \right) 
+ \int_0^T q_t^{\tilde \mu} \ell(t,\psi_t^{\tilde \mu}) \dd t 
\right] - {\mathcal S}(q^{\tilde \mu})
\\
&\hspace{15pt} + c 
{\mathbb E} \left[ \int_0^T
q_t^{\tilde \mu} \left(\vert \psi_t^{\tilde \mu} - \psi_t^{\mu} 
\vert^2 + \vert Y_t^{\star,\tilde \mu} - Y_t^{\star,\mu} \vert^2 
+ \vert Z_t^{\star,\tilde \mu} - Z_t^{\star,\mu} \vert^2 \right) \dd t \right].
\end{split}
\end{equation*}
By exchanging the roles of $\mu$ and $\mu'$ and by adding the resulting two inequalities, we obtain 
\begin{equation*} 
\begin{split}
0
&\geq
{\mathbb E} \left[ q_T^{\tilde \mu} \left( g(X_T^\mu,\mu) 
-
g(X_T^\mu,\tilde \mu) \right) 
 + q_T^{ \mu} \left( g(X_T^{\tilde \mu},\tilde \mu) 
-
g(X_T^{\tilde \mu},\mu) \right) 
\right]  
\\
&\hspace{15pt} + c 
{\mathbb E} \left[ \int_0^T
\left( q_t^{ \mu} 
+q_t^{\tilde \mu} 
\right) 
\left(\vert \psi_t^{\tilde \mu} - \psi_t^{\mu} 
\vert^2 + \vert Y_t^{\star,\tilde \mu} - Y_t^{\star,\mu} \vert^2 
+ \vert Z_t^{\star,\tilde \mu} - Z_t^{\star,\mu} \vert^2 \right) \dd t \right],
\end{split}
\end{equation*}
which completes the proof of \eqref{eq:corol:9:first:claim}. 
\vskip 4pt

\noindent \textit{Step 3}. 
We now derive the second claim. 
We come back to 
\eqref{corol:9:step1:end}. Exchanging the roles of 
$\mu$ and $\tilde \mu$ therein, and then summing the resulting two inequalities, 
we get 
\begin{equation*}
\begin{split}
&{\mathbb E} \left[ \left( q_T^\mu 
- q_T^{\tilde \mu} \right) \left( g(X_T^\mu,\mu) - g(X_T^{\tilde \mu},\tilde \mu) \right)
\right]
\\
&\geq
{\mathbb E} \left[  \left( q_T^{\mu} \nabla_x g(X_T^{\mu}, \mu) 
- q_T^{\tilde \mu} \nabla_x g(X_T^{\tilde \mu},\tilde \mu) \right) 
\cdot (X_T^{\mu} - X_T^{\tilde \mu})
\right]
\\
&\hspace{15pt} + c 
{\mathbb E} \left[ \int_0^T
\left( q_t^{\mu} + q_t^{\tilde \mu}\right) \left(\vert \psi_t^{\tilde \mu} - \psi_t^{\mu} 
\vert^2 + \vert Y_t^{\star,\tilde \mu} - Y_t^{\star,\mu} \vert^2 
+ \vert Z_t^{\star,\tilde \mu} - Z_t^{\star,\mu} \vert^2 \right) \dd t \right].
\end{split}
\end{equation*}
This completes the proof. 
\end{proof}

\section{Mean-field games}
\label{se:mfg}

This section is devoted to the study of the mean-field game problem \eqref{pb:mfg}. In Subsection \ref{subse:mfg:def}, we define the notion of a mean-field game equilibrium and introduce the topology underlying the existence result, together with additional assumptions on the interaction mapping $g$. In Subsection \ref{subse:4.2}, we derive uniform estimates on $(q^\mu,\psi^\mu)$ with respect to $\mu$, which are required to apply Schauder’s fixed point theorem. Subsection \ref{subse:mfg:existence} is devoted to our main existence result, stated in Theorem \ref{thm:existence}. Finally, in Subsection \ref{subse:4.4}, we establish a uniqueness result under a joint flat non-increasing and displacement non-decreasing condition on the mapping $g$, as defined in Definition \ref{def:flat:displacement}; see Proposition \ref{prop:uniqueness}.

\subsection{Definition of an equilibrium}
\label{subse:mfg:def}

\begin{definition}
\label{def:mfg:equilibrium}
For $r$ as in 
\eqref{hyp:sigma}, we say that 
$\mu \in {\mathcal M}_{2-r}({\mathbb R}^n)$ is an equilibrium to the robust mean-field 
game set over 
\eqref{pb:optim-mfg}
if the unique saddle point 
$(q^\mu,\psi^\mu)$ of 
\eqref{pb:optim-mfg} satisfies 
\begin{equation*}
\mu = (q_T^\mu {\mathbb P})_{X_T^\mu}.
\end{equation*} 
\end{definition}

\paragraph{Topology.}
Below, we study existence and uniqueness separately. 
For this, 
we equip the space of non-negative measures with the narrow  topology, 
a sequence $(\mu_k)_{k \geq 1}$ in ${\mathcal M}({\mathbb R}^k)$ converging 
narrowly to some $\mu$ in ${\mathcal M}({\mathbb R}^n)$ if, for any bounded and continuous function 
$f$ on ${\mathbb R}^n$, it holds
\begin{equation*}
\lim_{k \rightarrow + \infty} \int_{{\mathbb R}^n} f(x) \dd \mu_k(x) 
= \int_{{\mathbb R}^n} f(x) \dd  \mu(x). 
\end{equation*} 
In fact, we are only interested in elements $\mu \in {\mathcal M}({\mathbb R}^n)$ whose mass
$\mu({\mathbb R}^n)$ is less than 
$\exp(\alpha T)$. The reason is that, for any $\mu \in {\mathcal M}({\mathbb R}^n)$, 
${\mathbb E}[q_T^\mu] \leq \exp(\alpha T)$. In this regard, it is important to remember that 
Prokhorov's theorem extends easily to non-negative
 measures with a mass less than a fixed constant:
 
 \begin{lemma}
 \label{lem:prokhorov}
 Let ${\mathcal C}$ be a subset of ${\mathcal M}({\mathbb R}^n)$ such that 
 \begin{equation*} 
 \sup_{\mu \in {\mathcal C}} \mu({\mathbb R}^n) < + \infty.
 \end{equation*} 
 Then, ${\mathcal C}$ is relatively compact for the narrow topology if 
 it is tight, i.e., for any $\varepsilon >0$, there exists a compact subset 
 $K \subset {\mathbb R}^n$ such that 
 \begin{equation*}
 \sup_{\mu \in {\mathcal C}} \mu({\mathbb R}^n \setminus K) \leq \varepsilon. 
 \end{equation*}
 \end{lemma}
 In what follows (see the forthcoming condition \ref{assumption:g:reg:flat}), we require the function $g$ to be continuous in $\mu$ with respect to 
 the narrow topology, but only on bounded subsets of ${\mathcal M}_{2-r}({\mathbb R}^n)$, i.e., 
 on subsets 
of the form
 \begin{equation*} 
{\mathcal B}_{{\mathcal M}_{2-r}}(C) :=
\left\{ \mu \in {\mathcal M}_{2-r}({\mathbb R}^n) : 
 \int_{{\mathbb R}^n} \left( 1 +\vert x \vert^{2-r}\right)  \dd \mu(x)  \leq C\right\}. 
 \end{equation*} 
 This notion is motivated by the following standard lemma:
 \begin{lemma}
 \label{lem:reg:narrow}
 Let $h : {\mathbb R}^n \rightarrow {\mathbb R}$ be a continuous function such that, for 
 some $c >0$ and $\eta \in (0,2-r)$, 
 $\vert h(x) \vert \leq c (1+ \vert x \vert^{2-r-\eta})$. Then, for any $C>0$,
  the function 
 \begin{equation*} 
 {\mathcal B}_{{\mathcal M}_{2-r}(C)} \ni \mu \mapsto \int_{{\mathbb R}^n} 
 h(x) \dd \mu(x)
 \end{equation*} 
 is continuous for the narrow topology. 
 \end{lemma}
As it is well-known, the result becomes false when 
$\eta=0$. In this case, continuity just holds but on subsets 
of 
${\mathcal M}_{2-r}({\mathbb R}^n)$ that are uniformly integrable. 
In our framework, we are not able to prove that, in full generality, the collection of measures 
$((q_T^\mu {\mathbb P})_{X_T^\mu})_{\mu \in {\mathcal M}_{2-r}({\mathbb R}^n)}$
is uniformly integrable, which explains why continuity with respect to the narrow topology is required on larger subsets (and thus leads to less general examples, as clearly illustrated by the above lemma).  
 
We thus require further regularity properties on the cost function 
$g$ with respect to the measure argument: 

\paragraph{Assumptions}\textit{(continued)}
\begin{enumerate}[label*=A\arabic*,resume]
\item \label{assumption:g:growth} 
For the same $L$ as in \ref{hyp:b}-\ref{hyp:g:mu:fixed}, 
condition 
\eqref{hyp:G-growth-JC}
holds true with $L_C=L$, for any $\mu \in {\mathcal M}_{2-r}({\mathbb R}^n)$ such that 
$\mu({\mathbb R}^n) \leq \exp(\alpha T)$. 

\item \label{assumption:g:reg:flat} 
For any $C>0$
and
for any sequence 
$(\mu_\ell)_{\ell \geq 1}$ in ${\mathcal B}_{{\mathcal M}_{2-r}}(C)$ that converges 
narrowly 
to some 
$\mu$, it holds
    \begin{equation} \label{ass:unif-flat-dG-mean-field}
   \lim_{\ell \rightarrow \infty}     \sup_{x \in \mathbb{R}^n}\left[ \frac1{1+\vert x \vert^{2-r}} \left|g(x,\mu_\ell)  -  g(x,\mu) \right|\right] =0.
    \end{equation}    \end{enumerate}

\begin{remark} 
\label{re:10:10}
The following comments are in order. 
\begin{itemize}

\item  Thanks to Lemma \ref{lem:prokhorov}, 
it is plain to see that, for a given $C>0$, the \textit{ball} 
$${\mathcal B}_{{\mathcal M}_{2-r}}(C) := \left\{\mu \in {\mathcal M}_{2-r}({\mathbb R}^n) : \int_{{\mathbb R}^n} 
\left( 1 + 
\vert x \vert \right) \dd \mu(x) \leq C\right\}$$
is 
relatively compact for the narrow topology. In fact, it is also closed and hence compact. 
In particular, the measure $\mu$ in 
\ref{assumption:g:reg:flat}
is necessarily in ${\mathcal B}_{{\mathcal M}_{2-r}(C)}$.

\item Following the above item, we notice that any real-valued function on ${\mathcal M}({\mathbb R}^n)$ that is continuous
on ${\mathcal B}_{{\mathcal M}_{2-r}}(C)$ 
with respect to the narrow topology, for some $C>0$, is in fact uniformly continuous. 
In particular, 
for each $x \in {\mathbb R}^n$, the function $\mu \mapsto g(x,\mu)$ is, under condition 
\ref{assumption:g:reg:flat}, 
uniformly continuous on 
${\mathcal B}_{{\mathcal M}_{2-r}}(C)$. 
Somehow, condition \ref{assumption:g:reg:flat} imposes an additional constraint on the modulus of continuity, but uniformly in $x$.

\item 
Following Lemma 
\ref{lem:reg:narrow},
a standard example of a function $g$ that satisfies all the requirements 
\ref{hyp:g:mu:fixed}-\ref{assumption:g:reg:flat} is 
\begin{equation*}
g(x,\mu) := \int_{{\mathbb R}^n} \Gamma(x,y) \dd \mu(y), 
\end{equation*}
where $\Gamma$ is convex in the variable $x$ and satisfies (all the derivatives below being implicitly assumed to exist), 
\begin{equation*} 
\begin{array}{rl}
- L \left( 1 + \vert x \vert^{1-r}\right) 
\leq
 \Gamma(x,y)  &\leq  L \left( 1 + \vert x \vert^{2-r} \right), 
\\[.5em]
 \vert \nabla_x \Gamma(x,y) \vert &\leq  L \left( 1 + \vert x \vert^{1-r} \right), 
\\[.5em]
 \vert \nabla_{y} \Gamma(x,y) \vert &\leq  L \left( 1 + \vert x \vert^{2-r}\right), 
\\[.5em]
 \vert \nabla^2_{x,x} \Gamma(x,y) \vert &\leq  L.
\end{array}
\end{equation*}  
The proof 
of 
\eqref{ass:unif-flat-dG-mean-field}
is as follows (the other conditions in \ref{hyp:g:mu:fixed}-\ref{assumption:g:reg:flat} are easily checked). Let $C>0$ and $\varepsilon >0$.  
By the first line above (together with the first item in the remark), we can find a compact subset $K \subset {\mathbb R}^n$ such that, for 
any $\mu \in {\mathcal B}_{{\mathcal M}_{2-r}}(C)$, 
\begin{equation*} 
\sup_{x \in {\mathbb R}^n} \left[ \frac1{1+\vert x\vert^{2-r}}
\left\vert 
\int_{{\mathbb R}^n} 
\Gamma(x,y) \dd \mu(y) - 
\int_{K} 
\Gamma(x,y) \dd \mu(y)
\right\vert \right] \leq \varepsilon. 
\end{equation*} 
By the penultimate point, the functions $(y \mapsto \Gamma(x,y)/(1+\vert x \vert^{2-r}))_{x \in {\mathbb R}^n}$ are equicontinuous on $K$. Therefore, we can approximate any of them, to any fixed accuracy for the sup norm on $K$, 
by a continuous function in a finite collection. The proof is then easily completed. 

\item 
Similar to \cite{DelarueLavigne}, the presentation is restricted to games in which only the terminal cost
has a mean-field structure. That said, we could also consider mean-field running cost with a separated form
    \begin{equation*}
        \ell'(t,\psi_t,\mu_t) = \ell(t,\psi_t) + c(X_t,\mu_t),
    \end{equation*}
    where $\mu_t$ is the marginal law of $X_t$ under the probability measure $q_T \mathbb{P}$ .

\end{itemize}
\end{remark}

In the rest of this subsection, Assumptions 
\ref{hyp:g:mu:fixed}-\ref{assumption:g:reg:flat} 
are in force.

\subsection{Entropy and moment estimates}
\label{subse:4.2}

In this subsection, we provide a series of bounds that are satisfied for any 
$\mu \in {\mathcal M}_{2-r}({\mathbb R}^n)$.

We start with the following lemma:
\begin{lemma}
\label{lem:apriori:S}
There exists a constant $C_1$, only depending on  the parameters in 
the standing assumptions, such that 
\begin{equation*}
\sup_{\mu \in {\mathcal M}_{2-r}({\mathbb R}^n)}
{\mathcal S}(q^\mu) \leq C_1. 
\end{equation*}
In particular, (up to a possibly new value of $C_1$)
\begin{equation*} 
\sup_{\mu \in {\mathcal M}_{2-r}({\mathbb R}^n)}
\sup_{t \in [0,T]}
{\mathbb E}[h(q_t^\mu)] \leq C_1. 
\end{equation*} 
\end{lemma}

\begin{proof}
The result is a direct consequence of \cite[Lemma 26]{DelarueLavigne}. 
The main point is to observe that the quantity 
${\mathcal G}(q^0,X_T^\psi)$ appearing in the first step of the proof is 
here equal to 
${\mathbb E}[q^0_T g(X_T^\psi,\mu)]$. By convexity of 
$g$ in the variable $x$ and then by 
condition \ref{assumption:g:growth},
it is greater than 
\begin{equation*}
\begin{split} 
{\mathbb E}[q^0_T g(X_T^\psi,\mu)]
&\geq {\mathbb E}[q^0_T g(0,\mu)]
+ {\mathbb E}[q^0_T \nabla_x g(0,\mu) \cdot X_T^\psi] 
\\
&\geq - L \left( 1 +{\mathbb E}[q_T^0 \vert X_T^\psi \vert] \right).
\end{split}
\end{equation*}
The key fact is that the constant $L$ is here independent of $\mu$. Inserting this bound in 
\cite[(131)]{DelarueLavigne}, we get a constant 
$C_1$ in 
\cite[(134)]{DelarueLavigne} that is independent 
of $\mu$. Following the rest of the proof in 
\cite{DelarueLavigne}, we deduce that 
$C_1$
in the statement can be chosen independently of $\mu$.
\end{proof}

\begin{lemma}
\label{lem:apriori:SStar}
There exists a constant $C_2$, only depending on  the parameters in 
the standing assumptions, such that 
\begin{equation*}
\sup_{\mu \in {\mathcal M}_{2-r}({\mathbb R}^n)}
{\mathcal S}^\star(\psi^\mu) \leq C_2. 
\end{equation*}
\end{lemma}

\begin{proof}
The proof is an adaptation of \cite[Lemma 32]{DelarueLavigne}. 
The bound established therein
depends on 
$g$ through $g$ and $\nabla_x g$ at $x=0$, but the latter two are bounded independently 
of 
$\mu$, see 
\ref {assumption:g:growth}. 

One also needs a bound for the cost driven by the null control. 
Thanks again to 
\ref {assumption:g:growth}, it is independent of 
$\mu$. The conclusion easily follows. 
\end{proof}

\begin{lemma}
\label{lem:apriori:X}
There exists a constant $C_3$, only depending on  the parameters in 
the standing assumptions, such that 
\begin{equation*}
\sup_{\mu \in {\mathcal M}_{2-r}({\mathbb R}^n)}
{\mathbb E}\left[ q_T^{\mu} \vert X_T^\mu\vert^{2-r} \right]
 \leq C_3. 
\end{equation*}
\end{lemma}

\begin{proof} 
This is a consequence of 
Lemmas \ref{lem:apriori:S} and \ref{lem:apriori:SStar}
and 
\ref{lemma:reg-X-psi-A}.
\end{proof}

\subsection{Existence}
\label{subse:mfg:existence}

Here is the first main result of the article. 

\begin{theorem}
\label{thm:existence}
Let Assumptions \ref{hyp:b}-\ref{assumption:g:reg:flat}  be in force. 
Then, there exists at least one equilibrium to the mean-field game set over 
\eqref{pb:optim-mfg}, in the sense of 
Definition 
\ref{def:mfg:equilibrium}. 
\end{theorem} 

\begin{proof}
The proof is an application of Schauder's theorem, see \cite[Corollary 17.56]{AliprantisBorder}. 
Throughout, we metricize the narrow topology introduced in Subsection 
\ref{subse:mfg:def} by means of the Fortet-Mourier distance. In fact, the latter extends to a norm on the whole space ${\mathcal M}_{\rm sign}({\mathbb R}^n)$ of signed measures on ${\mathbb R}^n$, 
given by 
\begin{equation*} 
\| \mu \|_{\textrm{\rm FM}} := 
\sup_{\varphi} \left[ \int_{{\mathbb R}^n} 
\varphi(x) \dd \mu(x) \right],
\end{equation*} 
where the supremum is taken over functions $\varphi$ that are bounded by $1$ and that are $1$-Lipschitz continuous. 

Given the  constant $C_3$ from 
Lemma 
\ref{lem:apriori:X}, we consider the collection ${\mathcal C}$ of
 measures 
$\mu \in {\mathcal M}_{2-r}({\mathbb R}^n)$
such that 
$\mu({\mathbb R}^n) \leq \exp(\alpha T)$ and 
$M_{2-r}(\mu) \leq C_3$. 
By Lemma 
\ref{lem:prokhorov} and similar the third point in Remark \ref{re:10:10}, we easily deduce that 
${\mathcal C}$ is compact for $\| \cdot \|_{\textrm{\rm FM}}$. 
Obviously, it is convex. 

We then consider the mappping 
\begin{equation*}
\Phi : {\mathcal C} \ni \mu \mapsto (q_T^\mu {\mathbb P})_{X_T^\mu}.
\end{equation*} 
By Lemma \ref{lem:apriori:X}, ${\mathcal C}$ is stable by 
$\Phi$. 

It remains to check that 
$\Phi$ is continuous. 
We thus consider a sequence $(\mu_\ell)_{\ell \geq 1}$ in ${\mathcal C}$ that converges to 
$\mu$ for the narrow topology. By closedness of ${\mathcal C}$, 
$\mu \in {\mathcal C}$. 
For simplicity, we write $(q,X,\psi,Y^\star,Z^\star)$ for 
$(q^\mu,X^\mu,\psi^\mu,Y^{\star,\mu},Z^{\star,\mu})$ and
$(q^\ell,X^\ell,\psi^\ell,Y^{\star,\ell},Z^{\star,\ell})$
for 
$(q^{\mu_\ell},X^{\mu_\ell},\psi^{\mu_\ell},Y^{\star,{\mu_\ell}},Z^{\star,\mu_\ell})$. 
By \eqref{eq:corol:9:first:claim}
in Proposition
\ref{prop:main:estimate}, 
there exists a constant $c>0$ such that, 
for any $\ell \geq 1$, 
\begin{equation*}
\begin{split}
&{\mathbb E} \left[ q_T^\ell \left( g(X_T,\mu^\ell) 
-
g(X_T,  \mu) \right) 
 + q_T  \left( g(X_T^{\ell},  \mu) 
-
g(X_T^{\ell},  \mu^\ell) \right) 
\right]  
\\
&\geq c 
{\mathbb E} \left[ \int_0^T
\left( q_t 
+q_t^\ell
\right) 
\left(\vert \psi_t^{\ell} - \psi_t
\vert^2 + \vert Y_t^{\star,\ell} - Y_t^{\star} \vert^2 
+ \vert Z_t^{\star,\ell} - Z_t^{\star} \vert^2 \right) \dd t \right]. 
\end{split}
\end{equation*}
By 
\ref{assumption:g:reg:flat}, there exists a 
sequence $\varepsilon_\ell$ that tends to $0$ such that 
\begin{equation*} 
\begin{split}
&{\mathbb E} \left[ q_T^\ell \left( g(X_T,\mu^\ell) 
-
g(X_T,  \mu) \right) 
 + q_T  \left( g(X_T^\ell,  \mu) 
-
g(X_T^{\ell},  \mu^\ell) \right) 
\right]  
\\
&\leq \varepsilon_\ell {\mathbb E} \left[ q_T^\ell \left ( 1+ \vert X_T \vert^{2-r} \right) 
 + q_T \left( 1+ \vert X_T^\ell \vert^{2-r} \right) \right]. 
 \end{split} 
\end{equation*} 
By Lemmas \ref{lem:apriori:S}, \ref{lem:apriori:SStar}
and \ref{lemma:reg-X-psi-A}, 
\begin{equation*} 
\sup_{\ell \geq 1} {\mathbb E} \left[ q_T^\ell \left ( 1+ \vert X_T\vert^{2-r} \right) 
 + q_T \left( 1+ \vert X_T^\ell \vert^{2-r} \right) \right] < + \infty.
\end{equation*}
By the last three displays we deduce that
\begin{equation}
\label{eq:schauder:0}
\begin{split}
\lim_{\ell \rightarrow + \infty} 
{\mathbb E} \left[ \int_0^T
\left( q_t 
+q_t^\ell
\right) 
\left(\vert \psi_t^{\ell} - \psi_t
\vert^2 + \vert Y_t^{\star,\ell} - Y_t^{\star} \vert^2 
+ \vert Z_t^{\star,\ell} - Z_t^{\star} \vert^2 \right) \dd t \right] = 0. 
\end{split}
\end{equation}
Using the linearity of the dynamics of $X$, we observe that 
\begin{equation*} 
X_T - X_T^\ell = X_T^{\psi - \psi^\ell}. 
\end{equation*}
And then, following the proof \cite[Lemma 41]{DelarueLavigne} (which corresponds to Lemma 
\ref{lemma:reg-X-psi-A}), we deduce that 
\begin{equation}
\label{eq:schauder:1} 
\lim_{\ell \rightarrow \infty} {\mathbb E}\left[ \left( q_T + q_T^\ell\right) 
\vert X_T - X_T^\ell \vert^{2-r} \right]=0. 
\end{equation} 
It then remains to prove that 
\begin{equation}
\label{eq:schauder:2} 
\lim_{\ell \rightarrow \infty} {\mathbb E}\left[ \vert q_T - q_T^\ell\vert\right]=0.
\end{equation}
Assume indeed that the above holds true. Then, by combining 
\eqref{eq:schauder:1}
and
\eqref{eq:schauder:2}, we obtain, for any test function 
$\varphi : {\mathbb R}^n \rightarrow {\mathbb R}$ that is bounded by 1
and 1-Lipschitz, 
\begin{equation*}
{\mathbb E} \left[ \vert q_T \varphi(X_T) - q_T^\ell \varphi(X_T^\ell) \vert \right]
\leq {\mathbb E} \left[ \vert q_T - q_T^\ell \vert\right]
+
{\mathbb E} \left[ q_T \vert X_T - X_T^\ell\vert\right]. 
\end{equation*} 
Since the right-hand side tends to $0$ (as $\ell$ tends to $\infty$)
and is independent of $\varphi$, this gives 
$\| (q_T {\mathbb P})_{X_T} - (q_T^\ell  {\mathbb P})_{X_T^\ell}
\|_{\textrm{\rm FM}} \rightarrow 0$ as $\ell$ tends to $\infty$, which yields the required continuity property. 

We now prove \eqref{eq:schauder:2}. We 
let ${\mathcal E}_T:={\mathcal E}_T(\int_0^\cdot Z_s^\star \cdot \dd W_s)$ and 
${\mathcal E}_T^\ell :={\mathcal E}_T(\int_0^\cdot Z_s^{\star,\ell} \cdot \dd W_s)$. By definition of $q_T$ and $q^\ell_T$, we have 
\begin{equation*} 
q_T = \exp\left( \int_0^T Y_t^\star \dd t\right) \mathcal{E}_T, 
\quad 
q_T^\ell = 
 \exp\left( \int_0^T Y_t^{\star,\ell} \dd t\right)\mathcal{E}^\ell_T.
 \end{equation*}
Therefore,
\begin{align}  \nonumber
    \mathbb{E}\left[|q_T - q_T^\ell| \right] \leq \; &\mathbb{E}\left[\mathcal{E}_T \left \vert \exp\left( \int_0^T Y_t^\star \dd t\right) - \exp\left( \int_0^T Y_t^{\star,\ell} \dd t\right) \right \vert  \right] \\[0.5em]  \nonumber
    & + \mathbb{E}\left[ \exp\left( \int_0^T Y_t^{\star,\ell} \dd t\right) \left \vert\mathcal{E}_T - \mathcal{E}^\ell_T \right \vert  \right] \\[0.5em] \nonumber
    \leq \; & \mathbb{E} \left[\mathcal{E}_T\left \vert \exp\left( \int_0^T Y_t^\star \dd t\right) - \exp\left( \int_0^T Y_t^{\star,\ell} \dd t\right) \right \vert  \right] \\[0.5em]
    & + \exp(\alpha T) \mathbb{E}\left[\left \vert\mathcal{E}_T - \mathcal{E}^\ell_T \right \vert  \right], \label{ineq:q-q-prime-Z-Y}
\end{align}
where we used the fact that $Y^{\star,\ell}$ is bounded by $\alpha$ in the last inequality. 

We first consider the  the first term on the last inequality \eqref{ineq:q-q-prime-Z-Y}. Since  $Y^\star$ and $Y^{\star,\ell}$ are bounded by $\alpha$, we have that 
\begin{align} \nonumber
    \mathbb{E} &\left[ {\mathcal E}_T \left \vert \exp\left( \int_0^T Y_t^\star \dd t\right) - \exp\left( \int_0^T Y_t^{\star,\ell} \dd t\right) \right \vert  \right] \\[0.5em] \nonumber
    & \leq
\exp(\alpha T)\mathbb{E} \left[{\mathcal E}_T  \int_0^T \vert   Y_t^\star - Y_t^{\star,\ell} \vert \dd t \right] \\[0.5em]
    & \leq \exp(2 \alpha T)\mathbb{E}\left[q_T \int_0^T \vert   Y_t^\star - Y_t^{\star,\ell} \vert^2 \dd t \right]^{1/2}, \label{ineq:Y-Y-l}
\end{align}
where the last line follows from Cauchy-Schwarz' inequality, the definition of $q$ and the boundedness of $Y^\star$ again.

We now turn to the second term in \eqref{ineq:q-q-prime-Z-Y}. 
By Pinsker's inequality, we know 
that there exists a (universal) constant $c_0$ such that 
\begin{equation*} 
{\mathbb E} \left[ \vert {\mathcal E}_T - {\mathcal E}_T^\ell\vert \right]
\leq c_0
\sqrt{\mathcal{H}\left( {\mathcal E}_T{\mathbb P} \vert {\mathcal E}_T^\ell {\mathbb P}\right)}, 
\quad \textrm{\rm where} \quad \mathcal{H}\left( {\mathcal E}_T{\mathbb P} \vert {\mathcal E}_T^\ell {\mathbb P}\right)
:= {\mathbb E} \left[ \ln\left( \frac{{\mathcal E}_T}{{\mathcal E}_T^\ell} \right)
{\mathcal E}_T \right].
\end{equation*} 
It is standard to prove that 
\begin{equation*} 
\begin{split}
{\mathbb E} \left[ \ln\left( \frac{{\mathcal E}_T}{{\mathcal E}_T^\ell} \right)
{\mathcal E}_T \right]
&= \frac12 {\mathbb E} \left[ {\mathcal E}_T \int_0^T \vert Z_t^\star - Z_t^{\star,\ell} \vert^2 
\dd t \right]
\\[0.5em]
&\leq \exp(\alpha T) {\mathbb E} \left[ q_T \int_0^T \vert Z_t^\star - Z_t^{\star,\ell} \vert^2 
\dd t \right]. 
\end{split}
\end{equation*}
And then, there exists a constant $C_0$, independent of $\ell$, such that 
\begin{equation}  \label{ineq:E-E-l-Z-Z-l}
{\mathbb E} \left[ \vert {\mathcal E}_T - {\mathcal E}_T^\ell\vert \right]
\leq  C_0 {\mathbb E} \left[ q_T \int_0^T \vert Z_t^\star - Z_t^{\star,\ell} \vert^2 
\dd t \right]^{1/2}.
\end{equation} 
Finally combining \eqref{ineq:q-q-prime-Z-Y}--\eqref{ineq:Y-Y-l}--\eqref{ineq:E-E-l-Z-Z-l}, we obtain that 
\begin{align}  \nonumber
    \mathbb{E}\left[|q_T - q_T^\ell| \right] \leq C \left( \mathbb{E}\left[q_T \int_0^T \vert   Y_t^\star - Y_t^{\star,\ell} \vert^2 \dd t \right]^{1/2} + {\mathbb E} \left[ q_T \int_0^T \vert Z_t^\star - Z_t^{\star,\ell} \vert^2 
\dd t \right]^{1/2} \right), 
\end{align}
for some $C >0$ independent on $\ell$, and the conclusion follows by \eqref{eq:schauder:0}.
\end{proof}

\subsection{Uniqueness criterion}
\label{subse:4.4}
Uniqueness is a more subtle issue than in standard mean-field games, due the presence of Nature. 
To understand this, we may just focus on the situation where $T$ is small. 
Of course, we want to use the stability inequality 
\eqref{corol:9:step2:end}, which we recall here for convenience: 
\begin{equation}
\label{corol:9:step2:for:uniqueness} 
\begin{split}
&{\mathbb E} \left[ \left( q_T^\mu 
- q_T^{\tilde \mu} \right) \left( g(X_T^\mu,\mu) - g(X_T^{\tilde \mu},\tilde \mu) \right)
\right]
\\
&\geq
{\mathbb E} \left[  \left( q_T^{\mu} \nabla_x g(X_T^{\mu}, \mu) 
- q_T^{\tilde \mu} \nabla_x g(X_T^{\tilde \mu},\tilde \mu) \right) 
\cdot (X_T^{\mu} - X_T^{\tilde \mu})
\right]
\\
&\hspace{15pt} + c 
{\mathbb E} \left[ \int_0^T
\left( q_t^{\mu} + q_t^{\tilde \mu}\right)  \left(\vert \psi_t^{\tilde \mu} - \psi_t^{\mu} 
\vert^2 + \vert Y_t^{\star,\tilde \mu} - Y_t^{\star,\mu} \vert^2 
+ \vert Z_t^{\star,\tilde \mu} - Z_t^{\star,\mu} \vert^2 \right) \dd t \right].
\end{split}
\end{equation}
In absence of Nature, this inequality becomes very much simpler and just writes
\begin{equation*}
\begin{split}
&
{\mathbb E} \left[  \left(   \nabla_x g(X_T^{\mu}, \mu) 
-  \nabla_x g(X_T^{\tilde \mu},\tilde \mu) \right) 
\cdot (X_T^{\mu} - X_T^{\tilde \mu})
\right]
 \geq c 
{\mathbb E} \left[ \int_0^T
 \vert \psi_t^{\tilde \mu} - \psi_t^{\mu} 
\vert^2  \dd t \right].
\end{split}
\end{equation*}
Although we do not pretend that the derivation of the above display is 
rigorous, it turns out that this is indeed what can be obtained by using the so-called 
`probabilistic approach to mean-field games', see for instance \cite[Chapter 4]{carmona2018probabilistic-v1}. 
When 
$\mu$ is understood as the law of $X_T$ and $\mu'$ as the law of $X_T'$ (under 
the common probability measure ${\mathbb P}$), 
the left-hand side can be upper bounded, under Lipschitz assumptions on the derivatives of 
$g$ (the Lipschitz constant being denoted by the generic letter $L$), by 
\begin{equation*}
{\mathbb E} \left[  \left(   \nabla_x g(X_T^{\mu}, \mu) 
-  \nabla_x g(X_T^{\tilde \mu},\tilde \mu) \right) 
\cdot (X_T^{\mu} - X_T^{\tilde \mu})
\right] \leq L T {\mathbb E} \left[ \int_0^T
 \vert \psi_t^{\tilde \mu} - \psi_t^{\mu} 
\vert^2  \dd t \right].
\end{equation*}
The extra factor $T$ on the right-hand side makes it possible to guarantee uniqueness in small time. 

Here, we want to argue, at least informally, that the same computation can not be reproduced 
in the robust setting. When 
$\mu$ and $\tilde \mu$ are understood as $(q_T {\mathbb P})_{X_T}$ and 
$(\tilde q_T {\mathbb P})_{\tilde X_T}$
respectively 
(with $(q_T,X_T)$ standing for $(q_T^\mu,X_T^\mu)$, and 
$(\tilde q_T,\tilde X_T)$ for $(q_T^{\tilde \mu},X_T^{\tilde \mu})$), the first term on 
\eqref{corol:9:step2:for:uniqueness}
can be estimated as follows, at least in the simpler situation where 
$X_T$ and $\tilde X_T$ are equal (which is of course not true in general, 
but which cannot make the difficulty worse). In the latter situation, we are led to estimate 
$g(X_T,(q_T {\mathbb P}_{X_T}))-g(X_T,(\tilde q_T {\mathbb P})_{X_T})$. 
At best, we can expect to upper bound it by 
${\mathbb E}[\vert q_T - \tilde q_T\vert]$. 
Therefore, the first term on 
\eqref{corol:9:step2:for:uniqueness}
can be bounded by 
${\mathbb E}[\vert q_T - \tilde q_T\vert]^2$, up to a multiplicative constant. 
Then, we know from 
the proof of Theorem \ref{thm:existence}
that this   term can be bounded by means of Pinsker inequality by 
${\mathbb E}[q_T \int_0^T \vert Z_t^\star - \tilde Z_t^\star \vert^2 \dd t]
$ (with an obvious meaning for $Z^\star$ and $\tilde Z^\star$). In particular, there is no extra factor $T$ that could render its contribution smaller than the contribution of the corresponding 
 term on the right-hand side of 
\eqref{corol:9:step2:for:uniqueness}. 

Of course, a more direct way to obtain uniqueness is to multiply $g$ by a small parameter and then obtain 
the desired `contraction' (in the sense that the right hand side on \eqref{corol:9:step2:for:uniqueness} 
dominates the left hand side when $\mu$ and $\tilde \mu$ are equilibria). Although this could be one result towards uniqueness, we feel better to follow another route. The main point is to focus on the difference 
\begin{equation*}
\begin{split}
&{\mathbb E} \left[ \left( q_T^\mu 
- q_T^{\tilde \mu} \right) \left( g(X_T^\mu,\mu) - g(X_T^{\tilde \mu},\tilde \mu) \right)
\right]
\\
&\hspace{15pt} - 
{\mathbb E} \left[  \left( q_T^{\mu} \nabla_x g(X_T^{\mu}, \mu) 
- q_T^{\tilde \mu} \nabla_x g(X_T^{\tilde \mu},\tilde \mu) \right) 
\cdot (X_T^{\mu} - X_T^{\tilde \mu})
\right],
\end{split}
\end{equation*}
when $\mu$ and $\tilde \mu$ satisfy the fixed point conditions
\begin{equation*} 
\mu = (q_T^\mu {\mathbb P})_{X_T^\mu}, 
\quad
\tilde \mu = (q_T^{\tilde \mu} {\mathbb P})_{X_T^{\tilde \mu}}. 
\end{equation*}

This prompts us to introduce the following definition: 
\begin{definition} 
\label{def:flat:displacement}
The function $g$ is said to be (jointly) flat non-increasing/displacement non-decreasing
if, 
for any non-negative-valued random variables $q$ and $q'$
satisfying ${\mathbb E}[q],{\mathbb E}[q'] \leq \exp(\alpha T)$, and for 
any ${\mathbb R}^n$-valued random variables $X$ and $X'$ satisfying 
${\mathbb E}[(q+q') (\vert X \vert^{2-r} + \vert X' \vert^{2-r})] < + \infty$, 
it holds
\begin{equation}
\label{eq:flat:displacement}
\begin{split}
&{\mathbb E} \left[ \left( q
- q'  \right) \left( g(X,(q {\mathbb P})_X) - g(X',
(q' {\mathbb P})_{X'}) \right)
\right]
\\
&\hspace{15pt} - 
{\mathbb E} \left[  \left( q  \nabla_x g(X , (q {\mathbb P})_X) 
- q'  \nabla_x g(X', (q' {\mathbb P})_{X'}) \right) 
\cdot (X  - X')
\right] \leq 0.
\end{split}
\end{equation}
\end{definition}

\begin{remark} 
The following remarks are in order. 
\begin{itemize}
\item It is easy to see that the 
property 
\eqref{eq:flat:displacement} only depends on the joint law 
of $(q,q',X,X')$ under ${\mathbb P}$. In particular, the property 
\eqref{eq:flat:displacement}
can be transferred from one probability space to another. 

In fact, since 
the space $(\Omega,{\mathcal F},{\mathbb P})$
is here equipped with a Brownian motion, we can construct, for any given 
law on $(0,+\infty) \times (0,+\infty) \times {\mathbb R}^n \times {\mathbb R}^n$, 
a 4-tuple $(q,q',X,X')$ having this law under 
${\mathbb P}$ (as we can `reconstruct' any random variable from a 
random variable 
with uniform distribution on $[0,1]$). This guarantees that, on any 
probability space, the above inequality is guaranteed for any random 
variables $q,q',X,X'$ (taking values in the required spaces, and satisfying the required integrability properties).
\item 
Choose $q=q'=1$ in \eqref{eq:flat:displacement}, and deduce that, for any 
${\mathbb R}^n$-valued random variables $X,X'$
satisfying ${\mathbb E}[\vert X\vert^{2-r}],{\mathbb E}[\vert X' \vert^{2-r} ] < + \infty$, 
\begin{equation*}
{\mathbb E} \left[  \left( \nabla_x g(X ,  {\mathbb P}_X) 
-   \nabla_x g(X', {\mathbb P}_{X'})\right) 
\cdot (X  - X')
\right] \geq 0,
\end{equation*}
which is the standard 
displacement monotonicity property. 
\item Choose now $X=X'$ in \eqref{eq:flat:displacement}, and deduce that, 
for any $q,q'$ with positive values, 
\begin{equation*}
{\mathbb E} \left[ \left( q
- q'  \right) \left( g(X,(q {\mathbb P})_X) - g(X,
(q' {\mathbb P})_{X}) \right)
\right] \leq 0.
\end{equation*}
Choose now $X$ as being uniformly distributed on a given domain Borel subset 
$A \subset {\mathbb R}^n$ with finite Lebesgue measure (denoted 
$\textrm{\rm Leb}_{n}(A)$)
 and then 
$(q,q',X)$ such that $q=
\textrm{\rm Leb}_{n}(A)
f(X)$ and $q'=\textrm{\rm Leb}_{n}(A)f'(X)$ for two non-negative functions $f$ and $f'$ with support included in $A$, and satisfying $\int_{{\mathbb R}^n} f(x) \dd x, \int_{{\mathbb R}^n} f'(x) \dd x \leq \exp(\alpha T)$ and 
$\int_{{\mathbb R}^n} \vert x\vert^{2-r} f(x) \dd x, \int_{{\mathbb R}^n} \vert x \vert^{2-r} f'(x) \dd x< +\infty$. The above inequality can be rewritten
\begin{equation*} 
\int_{A} \left( f(x) - f'(x) \right) \left( g(x, f  \textrm{\rm Leb}_{n}) - g(x,f'  \textrm{\rm Leb}_{n})
\right) \dd x \leq 0. 
\end{equation*} 
Obviously ${\mathbb R}^n$ can be substituted for $A$ in the above display. And, then, by a standard approximation argument (using the regularity of $g$ in the measure argument), we deduce that the 
inequality holds true for $f$, $f'$ with $\int_{{\mathbb R}^n} f(x) \dd x, \int_{{\mathbb R}^n} f'(x) \dd x \leq \exp(\alpha T)$. And then, approximating any (finite non-negative) measure on ${\mathbb R}^n$ by 
measures with densities, we deduce that, for any measures $m,m' \in {\mathcal M}_{2-r}({\mathbb R}^n)$, 
\begin{equation*} 
\int_{A}   \left( g(x, m) - g(x,m')
\right) \dd \left( m - m'\right)(x) \leq 0,
\end{equation*} 
which is an anti-Lasry-Lions monotonicity condition.  
\item
It is not clear to us whether a function that is non-increasing in the flat sense (as in the previous item) and, separately displacement non-decreasing (as in the penultimate item), is 
(jointly) flat non-increasing/displacement non-decreasing as in Definition \ref{def:flat:displacement}. 
\end{itemize}
%
%
%

\end{remark}

We provide below a canonical example of a function $g$ satisfying Definition 
\ref{def:flat:displacement}.

\begin{lemma} 
\label{lem:17}
Let ${\mathcal G}$ be a function from 
${\mathcal M}_{2-r}({\mathbb R}^n)$ that is flat concave and displacement convex, in the sense that, for 
all $\mu,\mu' \in {\mathcal M}_{2-r}({\mathbb R}^n)$, 
\begin{equation} 
\begin{split} 
&{\mathcal G}(\mu') \leq {\mathcal G}(\mu) + \int_{{\mathbb R}^n} \frac{\delta {\mathcal G}}{\delta \mu}(\mu,x) 
\dd \left( \mu' - \mu\right)(x),
\\
&{\mathcal G}(\mu') \geq {\mathcal G}(\mu) + 
\int_{{\mathbb R}^n \times {\mathcal R}^n} 
\partial_\mu {\mathcal G}(\mu,x) \cdot (y-x) \dd \pi(x,y),
\end{split}
\label{eq:convexity:concavity:22}
\end{equation}
where 
$\pi$ on the last term is a coupling between $\mu$ and $\mu'$, i.e. 
$\pi$ has $\mu$ as first marginal on ${\mathbb R}^n$ and 
$\mu'$ as second marginal. 

Then, the function 
\begin{equation*} 
(x,\mu) \mapsto g(x,\mu) := 
\frac{\delta {\mathcal G}}{\delta \mu}(\mu,x)
\end{equation*}
is jointly flat non-increasing/displacement non-decreasing. 
\end{lemma}
The notions of derivatives used in the statement are standard. In brief, 
the flat derivative 
$\delta {\mathcal G}/ \delta \mu$ is defined 
as 
$$\frac{\delta {\mathcal G}}{ \delta \mu}(\mu,x) =
\frac{\dd}{\dd \varepsilon}_{\vert \varepsilon=0+} 
{\mathcal G}\left( \mu + \varepsilon \delta_x \right), 
\quad \mu \in {\mathcal M}_{2-r}({\mathbb R}^n), \quad 
x \in {\mathbb R}^n,$$
and
the intrinsic derivative 
$\partial_\mu {\mathcal G}$ as
$$\partial_{\mu} {\mathcal G}(\mu,x) =
\nabla_x \frac{\delta {\mathcal G}}{ \delta \mu}(\mu,x), 
\quad \mu \in {\mathcal M}_{2-r}({\mathbb R}^n), \quad 
x \in {\mathbb R}^n.$$
Implicitly, the derivatives are required to be jointly continuous in $\mu$ and $x$, and to satisfy growth conditions ensuring the well-posedness of the two integrals in \eqref{eq:convexity:concavity:22}. Precise definitions and conditions, together with examples of flat-concave/displacement-convex functions, are provided in \cite[Subsection~4]{DelarueLavigne}.

\begin{proof}
By  \cite[Corollary 15]{DelarueLavigne}, 
the following two inequalities hold true for any $q,q'$ and 
any $X,X'$
as in Definition \ref{def:flat:displacement}:
\begin{equation} 
\label{example:potential:1}
\begin{split}
   &  {\mathcal G} \left(q {\mathbb P})_{X'}  \right) 
    \geq  
    {\mathcal G}( (q {\mathbb P})_X) + \mathbb{E}\left[q  \partial_{\mu} {\mathcal G}\left(
    ( q {\mathbb P})_X    ,X \right) \cdot (X'-X) \right],
\\
&   {\mathcal G}\left((q' {\mathbb P})_X\right) 
    \leq  {\mathcal G} ((q {\mathbb P})_X) + \int_{\mathbb{R}^n} \frac{\delta {\mathcal G}}{\delta \mu}\left((q {\mathbb P})_X,x\right) \dd \left[(q' {\mathbb P})_X - (q {\mathbb P})_X\right](x).
\end{split}
\end{equation} 
We rewrite the second line as
\begin{equation*} 
{\mathcal G}\left((q {\mathbb P})_X\right) 
    \geq  {\mathcal G} ((q' {\mathbb P})_X) + \int_{\mathbb{R}^n} \frac{\delta {\mathcal G}}{\delta \mu}\left((q {\mathbb P})_X,x\right) \dd \left[(q {\mathbb P})_X - (q' {\mathbb P})_X\right](x).
\end{equation*} 
And then, exchanging the roles of $(q,X)$ and of $(q',X')$ in the above inequality
and then using the first line of 
\eqref{example:potential:1},
\begin{equation*}
\begin{split}
{\mathcal G}\left((q' {\mathbb P})_{X'}\right) 
    &\geq  {\mathcal G} ((q {\mathbb P})_{X'}) + \int_{\mathbb{R}^n} \frac{\delta {\mathcal G}}{\delta \mu}\left((q' {\mathbb P})_{X'},x\right) \dd \left[(q' {\mathbb P})_{X'} - (q{\mathbb P})_{X'}\right](x)
\\
&\geq     {\mathcal G}( (q {\mathbb P})_X) + \mathbb{E}\left[q  \partial_{\mu} {\mathcal G}\left(
    ( q {\mathbb P})_X    ,X \right) \cdot (X'-X) \right]\\
    &\hspace{15pt} 
    + \int_{\mathbb{R}^n} \frac{\delta {\mathcal G}}{\delta \mu}\left((q' {\mathbb P})_{X'},x\right) \dd \left[(q' {\mathbb P})_{X'} - (q{\mathbb P})_{X'}\right](x).
\end{split}
\end{equation*} 
Next, we exchange once again the roles of $(q,X)$ and $(q',X')$ and then sum the two resulting inqualities. We get 
\begin{equation*}
\begin{split}
&- \mathbb{E}\left[
\left( 
q'  \partial_{\mu} {\mathcal G}\left(
    ( q' {\mathbb P})_{X'}    ,X' \right) 
-
q  \partial_{\mu} {\mathcal G}\left(
    ( q {\mathbb P})_X    ,X \right)\right)    
    \cdot (X'-X) \right] 
    \\
    &\hspace{15pt} 
    + \int_{\mathbb{R}^n} \frac{\delta {\mathcal G}}{\delta \mu}\left((q' {\mathbb P})_{X'},x\right) \dd \left[(q' {\mathbb P})_{X'} - (q{\mathbb P})_{X'}\right](x)
    \\
    &\hspace{15pt} 
    + \int_{\mathbb{R}^n} \frac{\delta {\mathcal G}}{\delta \mu}\left((q {\mathbb P})_{X},x\right) \dd \left[(q {\mathbb P})_{X} - (q'{\mathbb P})_{X}\right](x) \leq 0.
\end{split}
\end{equation*}
Letting 
$g(x,\mu) = [\delta {\mathcal G}/\delta \mu](\mu,x)$ as 
done in the statement, and recalling   that 
$\nabla_x g(x,\mu) = \partial_\mu {\mathcal G}(\mu,x)$, the above display can be rewritten 
as
\begin{equation*}
\begin{split}
&- \mathbb{E}\left[
\left( 
q \nabla_x g 
    \left(X,  (q{\mathbb P})_{X}    \right)
-
q'  \nabla_x g \left(X',
    ( q' {\mathbb P})_{X'}     \right)   
    \right) 
    \cdot (X-X') \right] 
    \\
    &\hspace{15pt} 
    + \int_{\mathbb{R}^n} g\left(x,(q' {\mathbb P})_{X'}\right) \dd \left[(q' {\mathbb P})_{X'} - (q{\mathbb P})_{X'}\right](x)
    \\
    &\hspace{15pt} 
    + \int_{\mathbb{R}^n} g\left(x,(q {\mathbb P})_{X} \right) \dd \left[(q {\mathbb P})_{X} - (q'{\mathbb P})_{X}\right](x) \leq 0,
\end{split}
\end{equation*}
and then 
\begin{equation*}
\begin{split}
&- \mathbb{E}\left[
\left( 
q \nabla_x g 
    \left(X,  (q{\mathbb P})_{X}    \right)
-
q'  \nabla_x g \left(X',
    ( q' {\mathbb P})_{X'}     \right)   
    \right) 
    \cdot (X-X') \right] 
    \\
    &\hspace{15pt} 
    + {\mathbb E} \left[ 
    (q'-q) \left( g\left(X',(q' {\mathbb P})_{X'}\right)
    - 
    g\left(X,(q {\mathbb P})_{X}\right) \right) \right]\leq 0,
\end{split}
\end{equation*}
which  completes the proof. 
\end{proof}

Here is now the main result of this section: 
\begin{proposition}
\label{prop:uniqueness}
Let Assumptions \ref{hyp:b}-\ref{assumption:g:reg:flat}  be in force. 
If further, the function $g$ is flat non-increasing/displacement non-decreasing, then there exists a unique equilibrium in the sense of 
Definition 
\ref{def:mfg:equilibrium}. 
\end{proposition}

\begin{proof}
The proof is a straightforward consequence of 
\eqref{corol:9:step2:end}. 
Consider indeed two equilibria,
denoted by $\mu,\tilde \mu$. 
Using the fact that 
\begin{equation*} 
\mu = (q_T^\mu {\mathbb P})_{X_T^\mu}, 
\quad 
\tilde \mu = (q_T^{\tilde \mu} {\mathbb P})_{X_T^{\tilde \mu}}, 
\end{equation*}
we get 
\begin{equation*} 
\begin{split}
&{\mathbb E} \left[ \left( q_T^\mu 
- q_T^{\tilde \mu} \right) \left( g\left(X_T^\mu, 
(q_T^\mu {\mathbb P})_{X_T^\mu}
\right) - g\left(X_T^{\tilde \mu},  (q_T^{\tilde \mu} {\mathbb P})_{X_T^{\tilde \mu}}\right) \right)
\right]
\\
&\hspace{15pt}
- 
{\mathbb E} \left[  \left( q_T^{\mu} \nabla_x g\left(X_T^{\mu}, 
(q_T^\mu {\mathbb P})_{X_T^\mu}
\right) 
- q_T^{\tilde \mu} \nabla_x g\left(X_T^{\tilde \mu},  (q_T^{\tilde \mu} {\mathbb P})_{X_T^{\tilde \mu}}
\right) \right) 
\cdot (X_T^{\mu} - X_T^{\tilde \mu})
\right]
\\
&\geq c 
{\mathbb E} \left[ \int_0^T
\left( q_t^{\mu} + q_t^{\tilde \mu}\right)  \left(\vert \psi_t^{\tilde \mu} - \psi_t^{\mu} 
\vert^2 + \vert Y_t^{\star,\tilde \mu} - Y_t^{\star,\mu} \vert^2 
+ \vert Z_t^{\star,\tilde \mu} - Z_t^{\star,\mu} \vert^2 \right) \dd t \right].
\end{split}
\end{equation*}
By 
\eqref{eq:flat:displacement}, the left-hand side is less than (or equal to) $0$, from 
which we deduce that the right-hand side is equal to $0$. 
\end{proof}

\paragraph{Constructing flat non-decreasing/displacement non-increasing functions}
The purpose of this paragraph is to provide a tractable condition ensuring that $g$ satisfies Definition~\ref{def:flat:displacement}, beyond the potential regime introduced in the statement of Lemma~\ref{lem:17}.

Typically, we require the function 
$g$ to be jointly convex in the flat sense, i.e. 
\begin{equation} 
\label{eq:examples:joint:convex}
g(x',\mu') \geq g(x,\mu) + 
\nabla_x g(x,\mu) \cdot (x'-x) 
+ \int_{{\mathbb R}^n}
\frac{\delta g}{\delta \mu}
(x,\mu,y) \dd \left( \mu' - \mu \right)(y),
\end{equation} 
for any $x,x' \in {\mathbb R}^n$ and $\mu,\mu' \in {\mathcal M}_{2-r}({\mathbb R}^n)$. 
Implicitly, the function $g$ is assumed to be differentiable (in the flat sense) with respect to 
the measure argument, and   the integral 
on the right-hand side is assumed to make sense. 

Back to 
Definition~\ref{def:flat:displacement}, the purpose is to upper bound the left-hand side on
\eqref{eq:flat:displacement}.
Thanks to \eqref{eq:examples:joint:convex},
we have 
\begin{equation*}
\begin{split}
&{\mathbb E} \left[  q
 \left( g(X,(q {\mathbb P})_X) - g(X',
(q' {\mathbb P})_{X'})
-
\nabla_x g(X , (q {\mathbb P})_X) 
\cdot (X  - X')
 \right)
\right]
\\
&= {\mathbb E} \left[  q
 \left( g(X,(q {\mathbb P})_X) - g(X',
(q' {\mathbb P})_{X'})
+
\nabla_x g(X , (q {\mathbb P})_X) 
\cdot (X'  - X)
 \right)
\right]
\\
&\leq 
 - {\mathbb E} \left[  q
 \int_{{\mathbb R}^n}
\frac{\delta g}{\delta \mu}
\left(X, (q {\mathbb P})_X,y \right) \dd \left(  
(q' {\mathbb P})_{X'}
-
(q {\mathbb P})_X
 \right)(y) \right]. 
\end{split}
\end{equation*} 
Exchanging the roles of $(q,X)$ and $(q',X')$
and summing the two resulting inequalities, we get 
\begin{equation*}
\begin{split}
&{\mathbb E} \left[ \left( q
- q'  \right) \left( g(X,(q {\mathbb P})_X) - g(X',
(q' {\mathbb P})_{X'}) \right)
\right]
\\
&\hspace{15pt} - 
{\mathbb E} \left[  \left( q  \nabla_x g(X , (q {\mathbb P})_X) 
- q'  \nabla_x g(X', (q' {\mathbb P})_{X'}) \right) 
\cdot (X  - X')
\right] 
\\
&\leq 
  {\mathbb E} \left[  
 \int_{{\mathbb R}^n}
\left( q' 
\frac{\delta g}{\delta \mu}
\left( X',(q' {\mathbb P})_{X'},y \right) 
-
q
\frac{\delta g}{\delta \mu}
\left(X, (q {\mathbb P})_X,y \right)
\right)  \dd \left(  
(q' {\mathbb P})_{X'}
-
(q {\mathbb P})_X
 \right)(y) \right]
 \\
 &= 
 \int_{{\mathbb R}^n}
 \left[
 \int_{{\mathbb R}^n} 
\frac{\delta g}{\delta \mu}
\left( z,(q' {\mathbb P})_{X'},y \right) 
\dd(q' {\mathbb P})_{X'}(z) 
\right]\dd \left(  
(q' {\mathbb P})_{X'}
-
(q {\mathbb P})_X
 \right)(y)
 \\
&\hspace{15pt}  -
 \int_{{\mathbb R}^n}
\left[ 
 \int_{{\mathbb R}^n}
\frac{\delta g}{\delta \mu}
\left( z,(q {\mathbb P})_{X},y \right) 
\dd(q {\mathbb P})_{X}(z)
\right]
\dd \left(  
(q' {\mathbb P})_{X'}
-
(q {\mathbb P})_X
 \right)(y). 
 \end{split}
\end{equation*} 
And then, in order to guarantee
\eqref{eq:flat:displacement},  it suffices to have 
\begin{equation}
\label{eq:negative:type}
\begin{split}
& \int_{{\mathbb R}^n}
 \left[
 \int_{{\mathbb R}^n} 
\frac{\delta g}{\delta \mu}
\left( z,\mu',y \right) 
\dd\mu'(z) 
-
 \int_{{\mathbb R}^n} 
\frac{\delta g}{\delta \mu}
\left( z,\mu,y \right) 
\dd\mu(z) 
\right]\dd \left(  
\mu' 
-
\mu
 \right)(y) \leq 0, 
 \end{split}
\end{equation} 
for any 
$\mu,\mu' \in {\mathcal M}_{2-r}({\mathbb R}^n)$. 

Here is a typical example:
\begin{lemma}
\label{lem:example:neg}
Let $K : {\mathbb R}^n \times {\mathbb R}^n \rightarrow {\mathbb R}$ be a smooth function, bounded with bounded derivatives of any order, 
of negative type, i.e., 
satisfying for any smooth 
function $h : {\mathbb R}^n \rightarrow {\mathbb R}$
with a compact support, 
\begin{equation}
\label{eq:neg:type}
\int_{{\mathbb R}^n \times {\mathbb R}^n} 
K(x,y) h(x) h(y) \dd x \dd y
\leq 0.
\end{equation} 
Then, the function
$g$ defined by  
\begin{equation*}
g(x,\mu) = \int_{{\mathbb R}^n \times {\mathbb R}^n} K(x,y) \dd \mu(y), 
\quad x \in {\mathbb R}^n, \ \mu \in {\mathcal M}_{2-r}({\mathbb R}^n), 
\end{equation*} 
satisfies 
\eqref{eq:negative:type}. 
\end{lemma}
\begin{proof}
The proof is quite obvious as 
the left-hand side on 
\eqref{eq:negative:type}
rewrites
\begin{equation*}
\begin{split}
&\int_{{\mathbb R}^n}
 \left[
 \int_{{\mathbb R}^n} 
\frac{\delta g}{\delta \mu}
\left( z,\mu',y \right) 
\dd\mu'(z) 
-
 \int_{{\mathbb R}^n} 
\frac{\delta g}{\delta \mu}
\left( z,\mu,y \right) 
\dd\mu(z) 
\right]\dd \left(  
\mu' 
-
\mu
 \right)(y)
 \\
 &= \int_{{\mathbb R}^n}
 \int_{{\mathbb R}^n} 
 K(z,y) 
\dd \left( \mu'- \mu \right) (z)  \dd \left(  
\mu' 
-
\mu
 \right)(y).
\end{split}
\end{equation*} 
By
\eqref{eq:neg:type}, it is easy to see that the left-hand side is negative. 
\end{proof}

\begin{example}
{ \ } 

\begin{itemize}
\item 
A first example for $K$ satisfying 
\eqref{eq:neg:type} is 
\begin{equation*} 
K(x,y) = - \phi(x) \phi(y), 
\end{equation*} 
for a smooth function $\phi : {\mathbb R}^n \rightarrow {\mathbb R}$. 
\item Another example is 
\begin{equation*}
K(x,y) = - \int_{{\mathbb R}^n} \phi(x,r) \phi(y,r) \dd \lambda(r),
\end{equation*}
where $\lambda$ is a compactly supported  positive finite measure on ${\mathbb R}^k$, and
$\phi$ is a smooth function from ${\mathbb R}^n \times {\mathbb R}^k$ to ${\mathbb R}$. 
\item The first two examples are symmetric in $(x,y)$, as a result of which the function 
$g$, as defined in Lemma
\ref{lem:example:neg}, derives from a potential. 

That said, any (smooth) function $K$ that is
anti-symmetric, i.e. $K(x,y)=-K(y,x)$, satisfies 
Lemma
\ref{lem:example:neg}.
\end{itemize}
\end{example}

Of course, the function 
$g$ defined in the statement of Lemma 
\ref{lem:example:neg}
does not satisfy the joint convexity condition
\eqref{eq:examples:joint:convex}. 
To make it jointly convex, we may add a function that is convex in 
the variable $x$. 
Following the examples constructed in \cite[Subsection 4]{DelarueLavigne}, we claim 

\begin{lemma}
Let $K$ be 
as in the statement of Lemma \ref{lem:example:neg}. 

\begin{enumerate}
\item If $r=0$, we can find $\lambda$ large enough such that the function 
\begin{equation*} 
g(x,\mu) = \frac{\lambda}2 \vert x \vert^2 + \int_{{\mathbb R}^n} K(x,y) \dd \mu(y)
\end{equation*}
satisfies 
\eqref{eq:examples:joint:convex}
and, therefore, is 
jointly 
flat non-decreasing/displacement non-increasing functions.
\item If $r=1$ and $K$ is compactly supported, 
we can find $\lambda$ large enough such that the function 
\begin{equation*} 
g(x,\mu) =  \lambda \left( 1 + \vert x \vert^2 \right)^{1/2}+ \int_{{\mathbb R}^n} K(x,y) \dd \mu(y)
\end{equation*}
satisfies 
\eqref{eq:examples:joint:convex}
and, therefore, is jointly 
flat non-decreasing/displacement non-increasing functions.
\end{enumerate}
\end{lemma} 

Here, the choice of the convex perturbation is adapted to the value of $r$, so that $g$ satisfies the required growth properties in \ref{hyp:g:mu:fixed}.

\section{Limiting theory}
\label{se:limiting}

In this section, we investigate the connection between the mean-field game problem \eqref{pb:mfg} and a finite-player game in which $N$ players interact with Nature. The model is presented in  Subsection \ref{sec:lln}. In Subsection \ref{subsec:game-N-player-1-Nature}, we establish an $\varepsilon_N$-Nash equilibrium result for the finite-player game.

\subsection{Game with N-competitive players vs. Nature}
\label{sec:lln}

\paragraph{A primer on the law of large numbers.}

The construction of the $N$-player game relies on the following variant of the law of large numbers:

\begin{lemma}
\label{lem:lln:2}
Let $(q,X)$ be a random variable with values in $(0,+\infty) \times {\mathbb R}^n$. 
Assume that ${\mathbb E}[q]=1$. 
Let $(q^i,X^i)_{i \geq 1}$ be an I.I.D. sequence with the law of $(q,X)$ as common distribution (the 
sequence being constructed on $(\Omega,{\mathcal F},{\mathbb P})$). Then, 
\begin{equation*}
\forall \varepsilon >0, \quad 
\lim_{N \rightarrow + \infty} 
{\mathbb E} 
\left[ \left( \prod_{i=1}^N q^i \right) 
{\mathds 1}_{\{
d_{\textrm{\rm FM}}\left( 
\frac1N \sum_{i=1}^N \delta_{X^i} , (q {\mathbb P})_X
\right) > \varepsilon\}}
\right]=0,
\end{equation*}
where 
$d_{\textrm{\rm FM}}$ is the Fortet-Mourier distance
$d_{\textrm{\rm FM}}(\mu,\nu)= \| \mu - \nu \|_{\textrm{\rm FM}}$, for 
$\mu,\nu \in {\mathcal P}({\mathbb R}^n)$, and
$
\| \cdot \|_{\textrm{\rm FM}}$ 
is defined in 
the proof of Theorem 
\ref{thm:existence}. 
\end{lemma}
This result says that the standard empirical measure converges, in probability 
under $q^1\ldots q^N {\mathbb P}$, to $(q {\mathbb P})_X$.

\begin{proof}
We consider an I.I.D. sequence $(\tilde X^i)_{i \geq 1}$ with common distribution 
$(q {\mathbb P})_X$ under 
${\mathbb P}$. It is easy to see that, for each $N \geq 1$, the law of 
$(X^1,\ldots,X^N)$ under $q^1\ldots q^N {\mathbb P}$ is equal to 
$[(q {\mathbb P})_X]^{\times N}$, which is also the law of $(\tilde X^1,\ldots,\tilde X^N)$ (but under 
${\mathbb P}$). In particular, 
\begin{equation*}
\begin{split}
\forall \varepsilon >0, \quad
&{\mathbb E} 
\left[ \prod_{i=1}^N q_i 
{\mathds 1}_{\{
d_{\textrm{\rm FM}}\left( 
\frac1N \sum_{i=1}^N \delta_{X^i} , (q {\mathbb P})_X
\right) > \varepsilon\}}
\right]
\\
&= {\mathbb P} \left( 
\left\{
d_{\textrm{\rm FM}}\left( 
\frac1N \sum_{i=1}^N \delta_{\tilde X^i} , (q {\mathbb P})_X
\right) > \varepsilon
\right\}
\right),
\end{split}
\end{equation*}
but the right-hand side tends to $0$, as a consequence of the law of large numbers. 
\end{proof}

\paragraph{Presentation of the game.}
Based on Lemma \ref{lem:lln:2}, 
we now construct a game with $N$ competitive players playing against Nature, whose asymptotic version corresponds to the game studied in Section 
\ref{se:mfg}. 
Due to the restriction on the mass of $q$ imposed in Lemma 
\ref{lem:lln:2}, Nature' state in \eqref{eq:q:explicit:factorization}
is assumed to be a Doléans-Dade exponential, i.e. 
$Y^\star \equiv 0$ in the mean-field game. 

We consider the  product space $(\Omega^{\times N},{\mathcal F}^{\times N},{\mathbb P}^{\times N})$, 
and 
we equip its $i$-th factor  with an ${\mathbb R}^d$-valued Brownian motion 
$W^i=(W^i_t)_{t \in [0,T]}$ and an initial condition $\eta^i$, $\eta^i$ and $W^i$ being independent. 
We assume that all the random variables $\eta^1,\ldots,\eta^N$ are identically distributed, the support of their common statistical law being bounded. 
We denote by ${\mathbb F}^N=({\mathcal F}_t^N)_{0 \le t \le T}$ the completion of the filtration generated by 
$(\eta^1,\ldots,\eta^N,W^1,\ldots,W^n)$.

Below the function $b^i$ is a copy of $b$ on the $i$-th factor of $\Omega^{\times N}$, i.e., for any $\omega^{(N)} = (\omega^1, \ldots ,\omega^N) \in \Omega^{\times N}$, the quantity $b^i(\omega^{(N)},t,x,\psi)$ depends on $\omega^{(N)}$ only through $\omega^i$ and is thus equal to $b^i(\omega^{i},t,x,\psi)$.
The functions $\sigma^i$,
$\ell^i$ and $f^{\star,i}$ are constructed from $b$, $\sigma$,
$\ell$ and $f^{\star}$ in the same way.

The admissible set of Nature, denoted by $\mathcal{Q}^{(N)}$, is the class of $\mathbb{F}^N$-progressively measurable, positive valued processes $(Q^N_t)_{t \in [0,T]}$ such that (compare with \eqref{eq:q:explicit:factorization})
\begin{align*}
        &\mathcal{S}^{(N)}(Q^N) < + \infty, & Q^N_t :=    \prod_{i=1}^N
{\mathcal E}_t \left(  \int_0^\cdot 
Z_s^{\star,i} \cdot \dd W_s^i \right), \quad t \in [0,T],
  \end{align*}
with $(Z^{\star,1},\ldots,Z^{\star,N})$
acting as Nature's control. The mapping $\mathcal{S}^N$ denotes the $N$-player generalized entropy counterpart, defined as follows 
\begin{equation*}
{\mathcal S}^{(N)}(Q^N) 
:= 
{\mathbb E}^{\times N} 
\left[ \int_0^T
Q^N_s
\left( 
\sum_{i=1}^N f^{\star,i}(s,Z_s^{\star,i})
\right) 
\dd s
\right]. 
\end{equation*}
Here, the function 
\begin{equation*}
\mathbb{R}^{d \times N} \ni \left(z^{\star,1}, 
\ldots,z^{\star,N}\right)
\mapsto 
\sum_{i=1}^N
f^{\star,i}(t,z^{\star,i})
\end{equation*}
is understood as the Fenchel-Legendre transform of the function 
\begin{equation*}
\mathbb{R}^{d \times N} \ni  \left(z^{1}. 
\ldots,z^{N}\right) \mapsto \sum_{i=1}^N f^i(t,z^i),
\end{equation*}
where, as before, 
$f^i(\omega^{(N)},t,z)$ is equal to 
$f(t,\omega^i,z)$ (with $f$ being now independent of $y$ as 
Nature's mass remains equal to 1).

The control and state processes to player $i \in \{1,\ldots,N\}$ are denoted by 
$\psi^i = (\psi^i_t)_{t \in [0,T]}$ and 
$X^i = (X_t^i)_{t \in [0,T]}$ respectively, both processes 
taking values in ${\mathbb R}^n$. 
When needed, we 
write $X^{i,\psi^i}$ to emphasize the fact that 
$X^i$ is controlled by $\psi^i$. 
Following \eqref{eq:intro:X}, 
the dynamics of
 $X^i$ write
\begin{equation*} 
\dd X_t^i =b^i(t,X_t^i,\psi_t^i) \dd t + \sigma^i(t,\psi_t^i) \dd W_t^i, \quad t \in [0,T], 
\quad X_0^i = \eta^i.
\end{equation*} 
The admissible set of each player $i \in \{1,\ldots,N\}$ is denoted by $\mathcal{A}^{(N)}$ (it does not depend on $i$) and consists  
in a class of  
    ${\mathbb F}^N$-progressively-measurable,
    ${\mathbb R}^n$-valued processes 
    $ \psi=(\psi_t)_{0 \le t \le T}$ such that
    \begin{equation*}
    \mathcal{S}^{\star,(N)}(\psi) < +\infty, \quad 
       \mathcal{S}^{\star,(N)}(\psi) \coloneqq \sup_{Q^N \in \mathcal{Q}^{(N)}} \left\{\mathbb{E}\left[ \int_0^T Q^N_s |\psi_s|^2 \dd s \right] - \gamma \mathcal{S}^{(N)}(Q^N)\right\}.
    \end{equation*}
Importantly, the parameter $\gamma$ remains unchanged and is thus independent of $N$.

Given a control $q \in \mathcal{Q}^N$ of Nature, the cost to player $i \in \{1,\ldots,N\}$ is defined as
\begin{equation*}
\begin{split} 
{\mathcal R}^i(Q^N,\psi^1,\ldots,\psi^N) 
 &: =
 \mathbb{E}^{\times N} \left[ Q^N_T
 \left(  
  g^i\left(X^{\psi^i}_T, \mu_T^{\psi^1,\ldots,\psi^N}
\right) 
\right)\right]
\\
&\hspace{15pt} + \mathbb{E}^{\times N} \left[  \int_0^T  Q^N_s  \ell^i(s,\psi_s^i)  \dd s \right], 
\end{split} 
\end{equation*}
where 
\begin{equation} \label{def:empirical-minus-i}
\mu^{\psi^1,\ldots,\psi^N}_T := \frac{1}{N} \sum_{i =1}^N \delta_{X_T^{i,\psi^i}}.
\end{equation}
The $N$ adversarial players are  also in competition 
against Nature, whose reward is given by 

\begin{equation}
\label{eq:mathcalJN}
\begin{split} 
{\mathcal J}^{(N)}\left(Q^N,\psi^1,\ldots,\psi^N\right) &:=
{\mathbb E}^{\times N}\left[ Q_T^N \sum_{i=1}^N h \left(X_T^{i,\psi^i}, \mu^{\psi^1,\ldots,\psi^N}_T\right)
\right] 
\\
&\hspace{15pt} 
+ 
{\mathbb E}^{\times N} \left[ \int_0^T Q_s^N  
 \sum_{i=1}^N \ell^i( s,\psi_s^i) \dd s
\right]
-
\mathcal{S}^{(N)} \left(  Q^N \right),
\end{split} 
\end{equation}
where $h : {\mathbb R}^n \times {\mathcal P}_{2-r}({\mathbb R}^n) \rightarrow {\mathbb R}$
is a 
function satisfying the same properties as $g$ (It may be equal to $g$, but not necessarily).

The intuition is as follows: when Nature is frozen, players act as a in `standard' $N$-game under the 
measure $q{\mathbb P}$; but Nature penalizes them by choosing the worst (standing from players' viewpoint) 
$q$ according to the reward ${\mathcal J}^{(N)}$.

In this framework, we have
\begin{definition}
\label{eq:Nash:def}
A tuple $(Q^N,(\psi^1,\ldots,\psi^N)) \in {\mathcal Q}^{(N)} \times [{\mathcal A}^{(N)}]^N$ is said to be a Nash equilibrium (over open loop controls) if, 
for any other tuple $(\tilde Q^N,(\tilde \psi^1,\ldots,\tilde \psi^N))$ in the same class, the following $N+1$ inequalities hold true:
\begin{equation*}
{\mathcal R}^{i}\left(Q^N,(\psi^1,\ldots,\psi^N)\right) \leq{\mathcal R}^{i}\left(Q^N,(\psi^1,\ldots,\psi^{i-1},\tilde \psi^i, \psi^{i+1},\ldots,\psi^N)\right), 
\end{equation*}
for 
$i=1,\ldots,N$, 
and
\begin{equation*}
{\mathcal J}^{(N)}\left(Q^N,(\psi^1,\ldots,\psi^N)\right) \geq{\mathcal J}^{(N)}\left(\tilde Q^N,(\psi^1,\ldots,\psi^N)\right).
\end{equation*}
\end{definition}

\subsection{Approximate Nash equilibria}\label{subsec:game-N-player-1-Nature}

\paragraph{Strategy induced by a mean-field equilibrium.}

Thanks to Theorem \ref{thm:existence}, we can consider one equilibrium to the mean-field game set over 
\eqref{pb:optim-mfg}. We denote $q^*$ the Nature equilibrium state, and $\psi^*$ the player equilibrium control. 
Both $q^*$ and $\psi^*$ are defined on the space $(\Omega,{\mathcal F},{\mathbb P})$. 
On the extended product space $(\Omega^{\times N},{\mathcal F}^{\times N},{\mathbb P}^{\times N})$, 
we let, for any 
$i=1,\ldots,N$, 
\begin{equation*}
q^{*,i}(\omega_1,\ldots,\omega_N)= q^*(\omega^i), \quad 
\psi^{*,i}(\omega_1,\ldots,\omega_N) = \psi^*(\omega^i), \quad 
(\omega_1,\ldots,\omega_N) \in \Omega^{\times N},
\end{equation*} 
which makes it possible to define 
\begin{equation*}
Q^{*,N} := \prod_{i=1}^N q^{i}. 
\end{equation*} 
Below, we write 
$X^{*,i}$ for 
$X^{i,\psi^{*,i}}$, 
and 
we represent $q^{*,i}$ in the form 
$q^{*,i}={\mathcal E}_{\cdot}(\int_0^\cdot 
Z_s^{\star,*,i} \cdot \dd W_s^i)$. 

The strategy constructed in this way is called a \textit{mean-field} strategy. It could also be referred to as a \textit{distributed} strategy in the following sense : 
\begin{definition}
A strategy $(Q=q^1\ldots q^N,(\psi^1,\ldots,\psi^N))$ is said to be distributed if, for 
each $i \in \{1,\ldots,N\}$, 
$(q^i,\psi^i)$ is $\sigma(\eta^i,W^i)$-measurable. 
\end{definition}

\paragraph{Players deviating from the mean-field equilibrium.}
The purpose of the paragraphs below is to show that, in a certain sense, the mean-field strategy is an \textit{approximate equilibrium} of the $N$-player game.
The whole analysis is carried out 
under the following assumption, which is rather stronger than 
\ref{hyp:g:mu:fixed}-\ref{assumption:g:reg:flat} 
but which suffices to illustrate our approach:

\paragraph{Assumptions}\textit{(continued)}
\begin{enumerate}[label*=A\arabic*,resume]
\item \label{assumption:g:N} 
The function $g$ is the sum of two functions
\begin{equation*}
g(x,\mu) = g_0(x) + g_1(x,\mu),
\end{equation*}
where $g_0$ satisfies \ref{hyp:g:mu:fixed}, and 
 $g_1$ also satisfies \ref{hyp:g:mu:fixed}, and is bounded and Lipschitz continuous in 
$(x,\mu)$, when ${\mathcal P}_{2-r}({\mathbb R}^n)$
is equipped with the Fortet-Mourier distance.
\end{enumerate}

It is easy to see that 
\ref{assumption:g:growth} 
and 
\ref{assumption:g:reg:flat} 
are necessarily satisfied under \ref{assumption:g:N}.

 In the statement below, we show that a player who unilaterally deviates
from the mean-field strategy can only expect a modest reduction
in their loss. This corresponds to the classical result in mean-field game theory.

\begin{lemma}
\label{lem:23}
Let \ref{hyp:b}--\ref{assumption:g:N} be in force.
Then, there exists a sequence $(\varepsilon_N)_{N \geq 1}$ converging to 
$0$ such that, for any 
$i=1,\ldots,N$, and any 
$\tilde \psi^i \in {\mathcal A}^{(N)}$, it holds
\begin{equation*}
\begin{split} 
{\mathcal R}^i\left(Q^{*,N},\psi^{*,1},\ldots,\psi^{*,i-1},\tilde \psi^i,\psi^{*,i+1},\ldots,\psi^{*,N}\right) 
 &\geq {\mathcal R}^i\left(Q^{*,N},\psi^{*,1},\ldots,\psi^{*,N} \right) 
 - \varepsilon_N.
 \end{split} 
\end{equation*}
\end{lemma}

\begin{proof}
Throughout the proof, we use the convenient notation 
$\tilde X^i := X^{i,\tilde \psi^i}$. 

By Assumption 
\ref{assumption:g:N}, we can find a constant $C$, independent of $N$, such that, for any $i \in \{1,\ldots,N\}$. 
\begin{equation*}
\begin{split}
&\left\vert \mathbb{E}^{\times N} \left[ Q_T^{*,N}
 \left(  
  g\left(\tilde X^{i}_T,\mu_T^{\psi^{*,1},\ldots,\tilde \psi^{i},\ldots,\psi^{*,N}}
\right) 
-
  g\left(\tilde X^{i}_T, (q^* {\mathbb P})_{X^*}
\right) 
\right)\right] \right\vert
\\
&\leq C \mathbb{E}^{\times N} \left[ Q_T^{*,N}
\min\left(1 , 
d_{\textrm{\rm FM}}\left( 
\frac1N \sum_{j \not =i}\delta_{X^{*,j}_T} + \frac1N \delta_{\tilde X^i_T}, (q^* {\mathbb P})_{X^*}
\right)\right)
 \right].
\end{split} 
\end{equation*}
By Lemma
\ref{lem:lln:2}, the term on the second line tends to $0$ as $N$ tends to $+ \infty$. 
Therefore, we can find a sequence $(\varepsilon_N)_{N \geq 1}$, independent of 
$\tilde \psi^i$, such that 
\begin{equation*}
\begin{split} 
&{\mathcal R}^i\left(Q^{*N},\psi^{*,1},\ldots,\psi^{*,i-1},\tilde \psi^i,\psi^{*,i+1},\ldots,\psi^{*,N}\right) 
\\
&\geq  \mathbb{E}^{\times N} \left[ Q_T^{*,N}
 \left(  
  g^i\left(\tilde X^{i}_T, (q^* {\mathbb P})_{X^*}
\right) 
\right)\right]
 + \mathbb{E}^{\times N} \left[  \int_0^T  Q_s^{*,N}  \ell^i(s,\tilde \psi_s^i)  \dd s \right]
 - \varepsilon_N. 
 \end{split}
 \end{equation*} 
By convexity properties of the function
$g^i$ in the space variable and of the function 
$\ell^i$ in the variable $\psi$, we get 
\begin{equation*}
\begin{split} 
&{\mathcal R}^i\left(Q^{*,N},\psi^{*,1},\ldots,\psi^{*,i-1},\tilde \psi^i,\psi^{*,i+1},\ldots,\psi^{*,N}\right)
\\
&\geq  \mathbb{E}^{\times N} \left[ Q_T^{*,N}
 \left(  
  g\left(X_T^{*,i},(q^* {\mathbb P})_{X^*}
\right) 
\right)\right]
  + \mathbb{E}^{\times N} \left[  \int_0^T  Q_s^{*,N}  \ell^i(s, \psi_s^{*,i})  \dd s \right]
 \\
 &\hspace{15pt} 
+
\mathbb{E}^{\times N} \left[ Q_T^{*,N}
\nabla_x  
  g\left(X^{*,i}_T, (q^* {\mathbb P})_{X^*}
\right) \cdot (\tilde X^{i}_T - X_T^{*,i} ) \right]
\\
&\hspace{15pt} 
+
\mathbb{E}^{\times N} \left[ 
 \int_0^T  Q_s^{N} \nabla_\psi \ell^i(s, \psi_s^{*,i}) 
\cdot 
(\tilde \psi_s^i 
-
\psi_s^{*,i})
 \dd s \right] - \varepsilon_N.
 \end{split} 
\end{equation*}
The rest of the proof is quite standard and just consists in verifying that the strategy $\psi^{*,i}$ is optimal. The only difficulty is that
the process $\tilde \psi^i$ is defined on the product space $\Omega^{\times N}$. A careful inspection shows that the proof of 
\cite[Lemma 38]{DelarueLavigne}, which is based on It\^o calculus arguments and from which we already derived Lemma 
\ref{lemma:sufficient}, remains the same. This shows that the sum of the third and fourth terms on the right-hand side is equal to 
$0$.

Reverting the computations, we deduce that 
\begin{equation*}
\begin{split} 
&{\mathcal R}^i\left(Q^{*,N},\psi^{*,1},\ldots,\psi^{*,i-1},\tilde \psi^i,\psi^{*,i+1},\ldots,\psi^N \right)
 \geq {\mathcal R}^i\left(Q^{*,N},\psi^{*,1},\ldots,\psi^{*,N} \right)  - \varepsilon_N.
 \end{split} 
\end{equation*}
This completes the proof. 
\end{proof}

\paragraph{Mimicking the empirical distribution under 
$\tilde Q^{N} \in {\mathcal Q}^{(N)}$}

Deviations by Nature are more difficult to understand, due to the multiple correlations that may arise when modifying
$Q^{*,N}$. To overcome this difficulty, we rely on a rewriting of the cost function, whose principle is as follows and applies only to distributed strategies.

\begin{lemma}
\label{lem:LLN:polynomials}
Let the cost $h$ in 
\eqref{eq:mathcalJN}
satisfy 
\ref{hyp:g:mu:fixed} and \ref{assumption:g:N}, and let $c>0$. Then, there exists a sequence $(\varepsilon_N)_{N \geq 1}$ converging to 
$0$ such that, for any 
distributed strategy
$(\tilde Q^N,(\tilde \psi^1,\ldots,\tilde \psi^N)) \in {\mathcal Q}^{(N)} \times [{\mathcal A}^{(N)}]^N$ satisfying $$\sup_{i=1,\ldots,N} {\mathcal S}(\tilde q^i) \leq c, \quad \sup_{i=1,\ldots,N} {\mathcal S}^{\star}(\tilde \psi^i) \leq c,$$ it holds
\begin{equation*}
\begin{split} 
&\left\vert \mathbb{E}^{\times N} \left[ \tilde Q_T^{N}
\frac1N \sum_{i=1}^N
   h \left( \tilde X_T^i,  \mu_T^{\tilde \psi^1,\ldots,\tilde \psi^N}
\right) 
 \right] \color{white} \right\vert \color{black}
\\
&\hspace{15pt} 
 \color{white} \left\vert \color{black}
-
\frac1N \sum_{i=1}^N \mathbb{E}^{\times N} \left[ \tilde q_T^i
  h \left( \tilde X_T^i,
 \frac1{N} \sum_{j =1}^N \tilde q_T^j \delta_{\tilde X_T^j} 
\right)\right]
\right\vert \color{black} \leq \varepsilon_N.
 \end{split} 
\end{equation*}
\end{lemma}

\begin{proof}
Throughout the proof, $\epsilon$ denotes a fixed positive real. 
Moreover, 
we let $\varphi$ be a compactly supported function from ${\mathbb R}^n$ to 
${\mathbb R}$ that is equal to the identity on the ball $B_n(0,A)$ of center $0$ and of 
radius $A$, for a certain $A>0$, and satisfies $\vert \varphi(x) \vert \leq \vert x \vert$ for all 
$x \in {\mathbb R}^n$. With the shorthand notation
\begin{equation*} 
\hat X^i_T := \varphi(\tilde X^i_T), \quad i \in \{1,\ldots,N\}, 
\end{equation*}
we deduce from condition  
\ref{assumption:g:N} (for $h$) that, for a constant $C$ independent of $N$ and of 
$(\tilde Q^N,(\tilde \psi^1,\ldots,\tilde \psi^N))$, 
\begin{equation}
\label{eq:lln:5}
\begin{split}
&\left\vert {\mathbb E}^{\times N} \left[ \tilde Q_T^N h\left( \tilde X_T^1,\frac1{N} \sum_{i =1}^N \delta_{\tilde X_T^i}
\right) \right]  - {\mathbb E}^{\times N} \left[ \tilde Q_T^N h\left( \hat X_T^1,\frac1{N} \sum_{i =1}^N \delta_{\hat X^i_T}
\right) \right] \right\vert
\\
&\leq C 
{\mathbb E}^{\times N} \left[ \tilde Q_T^N \left(  1 \wedge \vert \tilde X_T^1 - \hat X_T^1 \vert  + \frac1{N} \sum_{i=1}^N 
1 \wedge \vert \tilde X_T^i - \hat X_T^i \vert \right) \right]
\\
&\leq  C {\mathbb E}^{\times N} \left[ \tilde Q_T^N  {\mathds 1}_{\{ \vert \tilde X_T^1 \vert \geq A\}} + \frac1{N} \sum_{i=1}^N 
\tilde Q_T^N  {\mathds 1}_{\{ \vert \tilde X_T^i \vert \geq A\}}  \right].
\end{split}
\end{equation}
Because the strategy 
$\tilde Q^N$ is distributed, we have
\begin{equation*} 
{\mathbb E}^{\times N} \left[   
\tilde Q_T^N  {\mathds 1}_{\{ \vert \tilde X_T^i \vert \geq A\}}  \right]
= {\mathbb E}^{\times N} \left[    
\tilde q_T^i  {\mathds 1}_{\{ \vert \tilde X_T^i \vert \geq A\}}  \right].
\end{equation*} 
Since ${\mathcal S}(\tilde q^i) \leq c$ and ${\mathcal S}^\star (\tilde \psi^i) \leq c$, 
we can choose $A$ large enough, only depending on 
$\epsilon$ and $c$, such that the right-hand side is less than $\epsilon/(2C)$; see Lemma 
\ref{lemma:reg-X-psi-A}. 

Next, we 
follow the proof of Lemma 
\ref{lem:lln:2}
and
consider an $N$-tuple $(Y_1,\ldots,Y_N)$ of independent random variables, 
constructed on $(\Omega,{\mathcal F},{\mathbb P})$, such that 
$Y_i \sim (\tilde q^i_T {\mathbb P}^{\times N})_{\hat{X}_T^i}$ for each 
$i \in \{1,\ldots,N\}$. 
We have 
\begin{equation*}
\begin{split}
&{\mathbb E}^{\times N} \left[ \tilde Q_T^N h\left( \hat X^1,\frac1{N} \sum_{i =1}^N \delta_{\hat X^i_T}
\right) \right]
 = {\mathbb E} \left[ h\left(Y^1,\frac1{N} \sum_{i =1}^N \delta_{Y^i}\right)
\right].  
\end{split}
\end{equation*} 
Following the standard $L^4$-proof of the  law of large numbers, we can find a universal constant $C$ such that, for any real-valued function $\varphi$  that is bounded by $1$ and $1$-Lipschitz continuous on the ball 
$B_n(0,A)$, 
\begin{equation}
\label{eq:lln:4}
{\mathbb E} \left[ \left\vert 
\frac1{N} \sum_{i =1}^N \varphi(Y^i) 
- 
\int_{{\mathbb R}^n} \varphi(x) \dd \mu^N(x)
\right\vert^4 \right]  \leq \frac{C}{N^2},
\end{equation} 
where 
\begin{equation*}
\mu^N = \frac1N \sum_{i=1}^N (\tilde q^i_T {\mathbb P}^{\times N})_{\hat{X}_T^i}. 
\end{equation*}
Call now $(\varphi_k)_{k \geq 1}$ a sequence that is dense (for the sup norm topology on the ball $B_n(0,A)$) in the set of real-valued functions that are bounded by $1$ and $1$-Lipschitz continuous on the ball 
$B_n(0,A)$. 
We deduce from 
the above bound (together with Markov inequality) that, ${\mathbb P}$-a.s., 
\begin{equation*}
\lim_{N \rightarrow \infty} \left\vert \frac1{N} \sum_{i =1}^N \varphi_k(Y^i) 
-
\int_{{\mathbb R}^n} \varphi_k(x) \dd \mu^N(x) \right\vert=0.
\end{equation*} 
And then, using the compactness 
of the collection of real-valued functions on $B_n(0,A)$ that are bounded by $1$ and $1$-Lipschitz continuous, we deduce that, ${\mathbb P}$-a.s.,  
\begin{equation*}
\lim_{N \rightarrow \infty} d_{\textrm{\rm FM}} 
\left( \frac1{N} \sum_{i =1}^N \delta_{Y^i},   \mu^N  \right)=0.
\end{equation*} 
Since the constant $C$ in 
\eqref{eq:lln:4}
is universal, it is easy to see that 
the rate is independent of 
$\tilde Q^N$, in the sense that the rate at which 
the sequence 
\begin{equation*}
{\mathbb P} \left( \left\{ 
d_{\textrm{\rm FM}} 
\left( \frac1{N} \sum_{i =1}^N \delta_{Y^i},   \mu^N  \right) > \varepsilon \right\}
\right) 
\end{equation*}
tends to $0$, for any given $\varepsilon >0$, is independent of $\tilde Q^N$. 

Combining with 
\eqref{eq:lln:5}, 
we deduce that there exists a sequence 
$(\varepsilon_N)_{N \geq 1}$, as in the statement, but depending on $A$, such that 
\begin{equation*}
\begin{split}
&\left\vert {\mathbb E}^{\times N} \left[ \tilde Q_T^N h\left( \tilde X_T^1,\frac1{N} \sum_{i =1}^N \delta_{\tilde X_T^i}
\right) \right]  - {\mathbb E}  \left[     h\left( Y^1,\mu^N
\right) \right] \right\vert\leq \epsilon+ \varepsilon_N,
\end{split}
\end{equation*} 
which we rewrite as
\begin{equation*}
\begin{split}
&\left\vert {\mathbb E}^{\times N} \left[ \tilde Q_T^N h\left( \tilde X_T^1,\frac1{N} \sum_{i =1}^N \delta_{\tilde X_T^i}
\right) \right]  - {\mathbb E}^{\times N}  \left[ \tilde q_T^1   h\left( \hat X_T^1,\mu^N
\right) \right] \right\vert\leq \epsilon + \varepsilon_N.
\end{split}
\end{equation*} 
For a given $a>0$, 
we now let 
\begin{equation*} 
\hat q^i_T := q^i_T {\mathds 1}_{\{ \tilde q^i_T \leq a\}}, \quad i =1,\ldots,N,
\end{equation*}
and then,
\begin{equation*}
\hat{\mu}^N := \frac1N \sum_{i=1}^N (\hat q^i_T {\mathbb P}^{\times N})_{\hat{X}_T^i}. 
\end{equation*}
We have
\begin{equation*}
d_{\textrm{\rm FM}}
\left( \hat{\mu}^N,\mu^N \right) \leq  \frac1N \sum_{i=1}^N
{\mathbb E}^{\times N} 
\left[ \tilde q^i_T {\mathds 1}_{\{ 
\tilde q^i_T \geq a\}}
\right].
\end{equation*} 
And, 
by 
\eqref{eq:H:S}, 
we can choose $a$ large enough, only depending on $c$, such that the right-hand side is 
less than $\epsilon/C$, where $C$ is the Lipschitz constant of $h$ with respect to  the
Fortet-Mourier distance. This shows that 
\begin{equation*}
\begin{split}
&\left\vert {\mathbb E}^{\times N} \left[ \tilde Q_T^N h\left( \tilde X_T^1,\frac1{N} \sum_{i =1}^N \delta_{\tilde X_T^i}
\right) \right]  - {\mathbb E}^{\times N}  \left[ \tilde q_T^1   h\left( \hat X_T^1,\hat \mu^N
\right) \right] \right\vert\leq 2 \epsilon + \varepsilon_N.
\end{split}
\end{equation*} 
Following 
\eqref{eq:lln:4}, we can find another constant, still denoted by $C$, only depending on $a$, such that, for any real-valued function $\varphi$ on the ball 
$B_n(0,A)$, bounded by $1$ and $1$-Lipschitz continuous, 
\begin{equation}
\label{eq:lln:4b}
{\mathbb E} \left[ \left\vert 
\frac1{N} \sum_{i =1}^N \hat{q}^i_T \varphi(\hat{X}_T^i) 
- 
\int_{{\mathbb R}^n} \varphi(x) \dd \hat{\mu}^N(x)
\right\vert^4 \right]  \leq \frac{C}{N^2}. 
\end{equation} 
Proceeding as before, 
this shows that, for any $\varepsilon >0$, 
the sequence 
\begin{equation*}
{\mathbb P} \left( \left\{ 
d_{\textrm{\rm FM}} 
\left( \frac1{N} \sum_{i =1}^N \hat{q}^i_T \delta_{\hat{X}_T^i},  \hat{\mu}^N  \right) > \varepsilon \right\}
\right) 
\end{equation*}
tends to $0$, with a rate that is independent of $\tilde Q^N$. 
And then, 
for a possibly new choice of 
the sequence $(\varepsilon_N)_{N \geq 1}$, 
\begin{equation*}
\begin{split}
&\left\vert {\mathbb E}^{\times N} \left[ \tilde Q_T^N h\left( \tilde X_T^1,\frac1{N} \sum_{i =1}^N \delta_{\tilde X_T^i}
\right) \right]  - {\mathbb E}^{\times N}  \left[ \tilde q_T^1   h\left( \hat X_T^1, \frac1{N} \sum_{i =1}^N \hat{q}^i_T \delta_{\hat{X}_T^i}
\right) \right] \right\vert\leq 2 \epsilon + \varepsilon_N. 
\end{split}
\end{equation*} 
It remains to see from condition  
\ref{assumption:g:N} that, for a new constant $C$, 
\begin{equation*}
\begin{split}
&\left\vert {\mathbb E}^{\times N} \left[\tilde q_T^1   h\left( \hat X_T^1, \frac1{N} \sum_{i =1}^N \tilde{q}^i_T \delta_{\hat{X}_T^i}
\right)  \right]  - {\mathbb E}^{\times N}  \left[ \tilde q_T^1   h\left( \hat X_T^1, \frac1{N} \sum_{i =1}^N \hat{q}^i_T \delta_{\hat{X}_T^i}
\right) \right] \right\vert
\\
&\leq C  {\mathbb E}^{\times N} \left[\tilde q_T^1\min\left(1, \frac1N \sum_{i=1}^N \vert  \tilde{q}^i_T - 
 \hat{q}^i_T\vert \right)\right]
 \\
 &\leq  C {\mathbb E}^{\times N} \left[\tilde q_T^1\min\left(1, \frac1N \sum_{i=1}^N   \tilde{q}^i_T 
 {\mathds 1}_{\{ 
 \tilde{q}^i_T\geq a\}} \right)\right].
\end{split}
\end{equation*}
And, thanks to \eqref{eq:H:S}, we can increase the value of $a$, only in function of $c$, so that the right-hand side is less than 
$\epsilon$. This gives
\begin{equation*}
\begin{split}
&\left\vert {\mathbb E}^{\times N} \left[ \tilde Q_T^N h\left( \tilde X_T^1,\frac1{N} \sum_{i =1}^N \delta_{\tilde X_T^i}
\right) \right]  - {\mathbb E}^{\times N}  \left[ \tilde q_T^1   h\left( \hat X_T^1, \frac1{N} \sum_{i =1}^N \tilde{q}^i_T \delta_{\hat{X}_T^i}
\right) \right] \right\vert\leq 3 \epsilon + \varepsilon_N. 
\end{split}
\end{equation*} 
Arguing in the same way, we can substitute $\tilde X_T^i$ for 
$\hat{X}_T^i$ in the above display, assuming that $A$ is large enough and adding 
a new $\epsilon$ in the right-hand side. Substituting $(\tilde X^j_T,\hat X^j_T)$ for 
$(\tilde X^1_T,\hat X^1_T)$, for any $j=2,\ldots,N$, and averaging 
over the indices $j \in \{1,\ldots,N\}$, we complete the proof.
\end{proof}

Lemma~\ref{lem:LLN:polynomials}
leads us to introduce a \textit{surrogate reward} for nature:

\begin{definition}
Given $(Q^N,(\psi^1,\ldots,\psi^N)) \in {\mathcal Q}^{(N)} \times [{\mathcal A}^{(N)}]^N$, we define the
surrogate reward for nature as
\begin{equation*}
\begin{split}
{\mathcal J}_{\textrm{\rm surrog}}^{(N)}\left(Q^N,\psi^1,\ldots,\psi^N\right) &:=
{\mathbb E}^{\times N}\left[  \sum_{i=1}^N q_T^i h^i \left(X_T^{i,\psi^i}, \frac1{N} \sum_{j=1}^N 
q_T^j \delta_{X_T^{j,\psi^j}} 
\right)
\right] 
\\
&\hspace{15pt} 
+ 
{\mathbb E}^{\times N} \left[   
 \sum_{i=1}^N \int_0^T q_s^i \left( \ell^i( s,\psi_s^i)-  f^{\star,i}(s,Z_s^{\star,i})
 \right) 
\dd s
\right]. 
\end{split} 
\end{equation*} 
\end{definition}
Lemma~\ref{lem:LLN:polynomials}
ensures that, for distributed strategies
$(Q^N,(\psi^1,\ldots,\psi^N))$, 
the costs
${\mathcal J}_{\textrm{\rm surrog}}^{(N)}\left(Q^N,\psi^1,\ldots,\psi^N\right)$
and
${\mathcal J}^{(N)}\left(Q^N,\psi^1,\ldots,\psi^N\right)$
are asymptotically close as $N \to \infty$.
Although this result is restricted to distributed strategies, we focus below on the surrogate 
reward, even for non-distributed strategies. Implicitly, this leads to the construction of approximate Nash equilibria, but for the surrogate game. 
When the game is restricted to distributed strategies,  approximate equilibria of the surrogate game are also approximate equilibria of the original game. 

When the strategy derives from a mean-field equilibrium, the empirical measure in the 
surrogate reward is governed by the following form of large of large numbers, which can be established as in the second part of the proof of Lemma \ref{lem:LLN:polynomials}: 
\begin{lemma}
\label{lem:lln:1}
Let $(q^i,X^i)_{i \geq 1}$ be an I.I.D sequence with the law of $(q,X)$ as common distribution on 
$(0,+\infty) \times {\mathbb R}^n$(the 
sequence being constructed on $(\Omega,{\mathcal F},{\mathbb P})$), where ${\mathbb E}[q]=1$. Then, ${\mathbb P}$-almost surely, 
\begin{equation*}
\lim_{N \rightarrow + \infty} 
\frac1N \sum_{i=1}^N q^i \delta_{X^i} = (q {\mathbb P})_X,
\end{equation*}
the limiting being understood for the narrow convergence. 
\end{lemma}

\paragraph{Nature locally deviating from the mean equilibrium.} 
In this paragraph, we choose $h = g$ in
\eqref{eq:mathcalJN}.

Our goal is to show that, in this case, the surrogate reward cannot increase as a result of a
local deviation, that is, when only the weight $q^i$ corresponding to the noise $W^i$ to which player $i$ is subjected is modified.
Below, we denote by 
${\mathcal Q}[i]$
the collection of Doléans-Dade exponentials
$\tilde q^i$ of
the form 
$(\tilde q^i_t  :={\mathcal E}_t(\int_0^\cdot \tilde Z_s^{\star,i} \cdot \dd W_s^i))_{t \in [0,T]}$
with $\tilde Z^{\star,i}$
being ${\mathbb F}^N$-progressively measurable and satisfying 
\begin{equation*}
{\mathcal S}(\tilde q^i) 
:= {\mathbb E}
\int_0^T \tilde q^i_s 
f^{\star,i}(s,\tilde Z_s^{\star,i}) \dd s < + \infty.
\end{equation*}

\begin{lemma}
\label{lem:27}
Let $c>0$. There exists a sequence 
$(\varepsilon_N)_{N>0}$ converging to $0$ such that, for any 
$i \in \{1,\ldots,N\}$ and any $\tilde q^i \in {\mathcal Q}[i]$ satisfying 
${\mathcal S}(\tilde q^i) \leq c$, 
\begin{equation*}
\begin{split}
{\mathcal J}_{\textrm{\rm surrog}}^{(N)}\left(\tilde Q^{*,N},\psi^{*,1},\ldots,\psi^{*,N}\right) &\leq 
{\mathcal J}_{\textrm{\rm surrog}}^{(N)}\left(  Q^{*,N},\psi^{*,1},\ldots,\psi^{*,N}\right)
+N \varepsilon_N,
\end{split} 
\end{equation*} 
with 
\begin{equation*}
\tilde Q^N_T = q^{*,1}_T \ldots q^{*,i-1}_T \tilde q^i_T q^{*,i+1}_T \ldots q_T^{*,N}.  
\end{equation*}
\end{lemma}

\begin{proof}
By Lipschitz property \ref{assumption:g:N} of $g$, 
\begin{equation*}
\begin{split} 
& \mathbb{E}^{\times N} \left[ \tilde q_T^i
\left( 
  g\left( \tilde X_T^i,
 \frac1{N} \sum_{j \not =i}    q_T^{*,j} \delta_{X_T^{*,j}} + 
 \frac1N \tilde q_T^i 
 \delta_{X_T^{*,i}}
\right)
-
  g\left(\tilde X_T^i,
(q^* {\mathbb P})_{X_T^*}
\right)
\right) 
 \right]
\\
&\leq  C \mathbb{E}^{\times N} \left[ \tilde q_T^i
 \min \left( 1, d_{\textrm{\rm FM}}
\left(
 \frac1{N} \sum_{j=1}^N    q_T^{*,j} \delta_{X_T^{*,j}} , 
(q^* {\mathbb P})_{X_T^*}
\right)\right)  
\right]
\\
&\hspace{15pt} 
+ C \mathbb{E}^{\times N} \left[\min \left( 1, \frac1N 
\vert q_T^{*,i} - \tilde q_T^i \vert
\right)   
\right].
 \end{split} 
\end{equation*}
By Lemma 
\ref{lem:lln:1}, we know that ${\mathbb P}$-a.s., 
\begin{equation*}
\lim_{N \rightarrow \infty} 
d_{\textrm{\rm FM}}
\left(
 \frac1{N} \sum_{j =1}^N    q_T^{*,j} \delta_{X_T^{*,j}} , 
(q^* {\mathbb P})_{X_T^*}
\right)
=0. 
\end{equation*} 
Since ${\mathcal S}(\tilde q^i) \leq c$, we can
use 
\eqref{eq:H:S} to find a sequence   
$(\varepsilon_N)_{N>0}$ converging to $0$, only depending on 
$\tilde q^i$ via $c$, 
 such that
 \begin{equation*}
\begin{split} 
& \mathbb{E}^{\times N} \left[ \tilde q_T^i
\left( 
  g\left( \tilde X_T^i,
 \frac1{N} \sum_{j =1}^N    q_T^{*,j} \delta_{X_T^{*,j}} 
\right)
-
  g\left(\tilde X_T^i,
(q^* {\mathbb P})_{X_T^*}
\right)
\right) 
 \right]
 \leq  \varepsilon_N.
 \end{split} 
\end{equation*}
Of course, we can proceed similarly with the coordinates 
$k \not = i$. 

Hence, 
denoting by $\tilde Z^{\star,i}$ the representative of 
$\tilde q^i$, i.e. $\tilde q^i = {\mathcal E}_{\cdot}(\int_0^{\cdot} 
\tilde Z^{\star,i}_s \cdot \dd W_s^i)$, 
we obtain 
\begin{equation*}
\begin{split}
{\mathcal J}_{\textrm{\rm surrog}}^{(N)}\left(\tilde Q^N,\psi^{*,1},\ldots,\psi^{*,N}\right) 
&\leq   \mathbb{E}^{\times N} \left[ \tilde q_T^i
  g\left( X_T^{*,i},
(q_T^* {\mathbb P})_{X^*_T}
\right) \right] 
\\
&\hspace{15pt} + 
{\mathbb E}^{\times N} \left[   
  \int_0^T \tilde q_s^i \left( \ell^i( s,\psi_s^{*,i})-  f^{\star,i}(s,\tilde Z_s^{\star,i})
 \right) 
\dd s
\right] 
\\
&\hspace{15pt} + \sum_{j \not = i} 
\left\{
\mathbb{E}^{\times N} \left[   q_T^{*,j}
  g\left( X_T^{*,j},
(q_T^* {\mathbb P})_{X^*_T}
\right) \right] 
\color{white} \int_0^T \right\} \color{black}
\\
&\hspace{7pt} 
\color{white} \left\{ \color{black}
+ 
{\mathbb E}^{\times N} \left[   
  \int_0^T   q_s^{*,j} \left( \ell^j( s,\psi_s^{*,j})-  f^{\star,i}(s, Z_s^{\star,*,j})
 \right) 
\dd s
\right] \right\} \color{black} 
\\
&\hspace{15pt} + N \varepsilon_N. \end{split} 
\end{equation*}
To handle the first two terms on the right-hand side, we use the optimality of 
$q^{*,i}$. As in the proof of Lemma 
\ref{lem:23}, the main subtlety comes from the fact that the probability space 
is not supported by $\Omega$
but by 
$\Omega^{\times N}$. That said, we can apply the same It\^o expansion 
as in the proof of \cite[Lemma 30]{DelarueLavigne} (see in particular the last display in the proof) to show that 
\begin{equation*}
\begin{split}
&\mathbb{E}^{\times N} \left[ \tilde q_T^i
  g\left( X_T^{*,i},
(q_T^* {\mathbb P})_{X^*_T}
\right) \right] 
 + 
{\mathbb E}^{\times N} \left[   
  \int_0^T \tilde q_s^i \left( \ell^i( s,\psi_s^{*,i})-  f^{\star,i}(s,\tilde Z_s^{\star,i})
 \right) 
\dd s
\right] 
\\
&\leq \mathbb{E}^{\times N} \left[ q_T^{*,i}
  g\left( X_T^{*,i},
(q_T^* {\mathbb P})_{X^*_T}
\right) \right] 
 + 
{\mathbb E}^{\times N} \left[   
  \int_0^T  q_s^{*,i} \left( \ell^i( s,\psi_s^{*,i})-  f^{\star,i}(s,Z_s^{\star,*,i})
 \right) 
\dd s
\right],
\end{split}
\end{equation*}
from which we deduce that  
\begin{equation*}
\begin{split}
{\mathcal J}_{\textrm{\rm surrog}}^{(N)}\left(\tilde Q^N,\psi^{*,1},\ldots,\psi^{*,N}\right) 
&\leq    \sum_{j =1}^N \mathbb{E}^{\times N} \left[   q_T^{*,j}
  g\left( X_T^{*,j},
(q_T^* {\mathbb P})_{X^*_T}
\right) \right] 
\\
&\hspace{15pt} + 
{\mathbb E}^{\times N} \left[   
  \int_0^T   q_s^{*,j} \left( \ell^j( s,\psi_s^{*,j})-  f^{\star,i}(s, Z_s^{\star,*,j})
 \right) 
\dd s
\right] 
\\
&\hspace{15pt} + N \varepsilon_N. 
\end{split} 
\end{equation*}
Repeating the computations, but with $q^{*,i}$ substituted for 
$\tilde q^i$, we get 
\begin{equation*}
\begin{split}
{\mathcal J}_{\textrm{\rm surrog}}^{(N)}\left(\tilde Q^N,\psi^{*,1},\ldots,\psi^{*,N}\right) 
&\leq {\mathcal J}_{\textrm{\rm surrog}}^{(N)}\left(Q^{*,N},\psi^{*,1},\ldots,\psi^{*,N}\right)     + N \varepsilon_N. 
\end{split} 
\end{equation*}
This completes the proof. 
\end{proof}

\paragraph{Nature globally deviating when the game is potential} 

In this paragraph, we assume that there exists 
a smooth function $G : {\mathcal M}({\mathbb R}^n) \rightarrow {\mathbb R}$ such that 
\begin{equation*}
g_1(x,\mu) = \frac{\delta G}{\delta \mu}(\mu,x) := 
\frac{d}{d \varepsilon}\vert_{\varepsilon = 0+} 
G\left( \mu + \varepsilon \delta _x \right), \quad x \in {\mathbb R}^n. 
\end{equation*}
We still assume 
\ref{hyp:g:mu:fixed} and
\ref{assumption:g:N}.

We choose, as surrogate cost, 
\begin{equation*}
\begin{split}
{\mathcal J}_{\textrm{\rm surrog}}^{(N)}\left(Q^N,\psi^1,\ldots,\psi^N\right) &:=
{\mathbb E}^{\times N}\left[  
\sum_{i=1}^N q_T^i g_0(X_T^i) + 
N G \left( \frac1{N} \sum_{j=1}^N 
q_T^j \delta_{X_T^{j,\psi^j}} 
\right)
\right] 
\\
&\hspace{15pt} 
+ 
{\mathbb E}^{\times N} \left[   
 \sum_{i=1}^N \int_0^T q_s^i \left( \ell^i( s,\psi_s^i)-  f^{\star,i}(s,Z_s^{\star,i})
 \right) 
\dd s
\right]. 
\end{split} 
\end{equation*}

\begin{lemma}
\label{lem:31}
Assume that the function 
$G$ is flat concave on the cone of non-negative measures. 
Then, for any constant $c>0$, there exists a sequence 
$(\varepsilon_N)_{N \geq 1}$, converging to $0$ such that, for any   $\tilde Q^N \in {\mathcal Q}^{(N)}$, 
with 
$\sup_{i=1,\ldots,N} {\mathcal S}(\tilde q^i) \leq c$, 
\begin{equation*}
\begin{split}
{\mathcal J}_{\textrm{\rm surrog}}^{(N)}\left(\tilde Q^N,\psi^{*,1},\ldots,\psi^{*,N}\right) &\leq 
{\mathcal J}_{\textrm{\rm surrog}}^{(N)}\left( Q^{*,N},\psi^{*,1},\ldots,\psi^{*,N}\right)
+N \varepsilon_N,
\end{split} 
\end{equation*}

\end{lemma}

\begin{proof} By concavity of $G$, 
\begin{equation*}
\begin{split}
{\mathbb E}^{\times N} \left[ N G\left( \frac1N \sum_{j=1}^N \tilde q^j_T \delta_{X^{*,j}_T} 
\right) \right] 
&\leq {\mathbb E}^{\times N} \left[ N G\left( \frac1N \sum_{j=1}^N q^{*,j}_T \delta_{X^{*,j}_T} 
\right) \right] 
\\
&\hspace{5pt} + \sum_{i=1}^N {\mathbb E}^{\times N} \left[
\left( \tilde q_T^i - q_T^{*,i} \right) 
 \frac{\delta G}{\delta \mu} \left( \frac1N \sum_{j=1}^N q^{*,j}_T \delta_{X^{*,j}_T}, X_T^{*,i}  
\right) \right]. 
\end{split}
\end{equation*}
By proceeding as in the proof of Lemma 
\ref{lem:27}, we deduce that there exists a sequence 
$(\varepsilon_N)_{N \geq 1}$ converging to $0$ and only depending on 
$\tilde Q^N$ via $c$ such that 
\begin{equation*}
\begin{split}
{\mathbb E}^{\times N} \left[ N G\left( \frac1N \sum_{j=1}^N \tilde q^j_T \delta_{X^{*,j}_T} 
\right) \right] 
&\leq {\mathbb E}^{\times N} \left[ N G\left( \frac1N \sum_{j=1}^N q^{*,j}_T \delta_{X^{*,j}_T} 
\right) \right]  
\\
&\hspace{5pt} + \sum_{i=1}^N {\mathbb E}^{\times N} \left[
\left( \tilde q_T^i - q_T^{*,i} \right) 
 \frac{\delta G}{\delta \mu} \left( (q_T^* {\mathbb P})_{X_T^*}, X_T^{*,i}  
\right) \right] + N \varepsilon_N. 
\end{split}
\end{equation*} 
Once again 
by  \cite[Lemma 30]{DelarueLavigne}, 
\begin{equation*}
\begin{split}
&\sum_{i=1}^N {\mathbb E}^{\times N} \left[
\left( \tilde q_T^i - q_T^{*,i} \right) 
 \frac{\delta G}{\delta \mu} \left( (q_T^* {\mathbb P})_{X_T^*}, X_T^{*,i}  
\right) \right]
\\
&\leq - \sum_{i=1}^N {\mathbb E}^{\times N} \left[
\int_0^T 
\left( \tilde q_s^i - q_s^{*,i} \right) 
\ell^i(s,\psi_s^{*,i}) \dd s 
 \right]
 \\
&\hspace{15pt} - \sum_{i=1}^N {\mathbb E}^{\times N} \left[
\int_0^T 
\left( \tilde q_s^i f^{i,\star}(s,\tilde Z_s^{\star,i}) - q_s^{*,i} f^{i,\star}(s,Z_s^{\star,*,i}  ) 
\right)  \dd s 
 \right].
\end{split}
\end{equation*} 
Combining the last two displays, we easily complete the proof. 
\end{proof} 

\subsection*{Acknowledgment}
F. Delarue and P. Lavigne  acknowledge the financial support of the European Research Council (ERC) under the European Union's Horizon Europe research and innovation program (ELISA project, Grant agreement No. 101054746). Views and opinions expressed are however
those of the author(s) only and do not necessarily reflect those of the European Union or the
European Research Council Executive Agency. Neither the European Union nor the granting
authority can be held responsible for them.
\bibliographystyle{plain}
\bibliography{biblio}

\end{document}